\documentclass[a4paper,11pt, reqno]{amsart}
\usepackage{amsmath}
\usepackage[utf8]{inputenc}
\usepackage[english]{babel}
\usepackage[babel]{csquotes}
\usepackage[backend=bibtex, firstinits=true, maxbibnames=99, sorting=nyt]{biblatex}
\renewbibmacro{in:}{}
\usepackage[english]{varioref}  
\usepackage{hyperref}

\usepackage{amssymb, amsthm, amsfonts, braket, bm} 
\usepackage{mathrsfs, mathtools}
\usepackage{mathabx}

\usepackage{subfig, xcolor, booktabs, tabularx, graphicx, emptypage, tikz}
\usetikzlibrary{decorations.markings, arrows.meta}   

\usepackage{caption, microtype, setspace}                     
\mathtoolsset{showonlyrefs}

\usepackage{articolo}

\bibliography{toto}
\AtEveryBibitem{
 \clearlist{address}
 \clearfield{date}
\clearfield{eprintclass}
 \clearfield{isbn}
 \clearfield{issn}
 \clearlist{location}
 \clearfield{month}
 \clearfield{series}
 \clearfield{url}
 \clearfield{doi}

 \ifentrytype{book}{}{
  \clearlist{publisher}
  \clearname{editor}
 
 }
 
}

\newcommand{\Xvec}{\vec{X}}
\newcommand{\Mrom}{\boldsymbol{\mathrm{M}}}
\DeclareMathOperator{\dist}{d}

\providecommand{\PhaseMatrix}[1]{\begin{bmatrix} \cos({#1}) & \sin({#1})\,\mhDelta^{-\frac12} \\ -\sin({#1}) \mhDelta^\frac12 & \cos({#1}) \end{bmatrix}}

\newcommand{\StrichartzConstant}{\mathcal{A}_3}
\newcommand{\SpaceDim}{{3}}
\newcommand{\sReg}{{\frac12}}
\newcommand{\PhaseSpace}{{\Hcaldot^{1/2}}}

\newcommand{\fboldstar}{{{\Bd f}_\star}}

\newcommand{\fboldbot}{{{\Bd f}_\bot}}

\newcommand{\Lfour}{{L^4(\R^{1+\SpaceDim})}}

\DeclareMathOperator{\Span}{span}

\newcommand{\RPolar}{R}
\newcommand{\OmegaZero}{\Omega_0}

\newcommand{\Degeneracy}{N}

\newcommand{\obold}{{\Bd{0}}}
\newcommand{\gammabold}{{\Bd{\gamma}}}

\newcommand{\Svec}{{\vec{S}}}

\newcommand{\Pzero}{\Pcal_0}

\newcommand{\sqrtDelta}{\sqrt{-\Delta}}

\newcommand{\mhDelta}{\Tonde{-\Delta}}

\DeclareMathOperator{\SO}{SO}

\title[A sharpened Strichartz inequality for the wave equation]{A sharpened Strichartz inequality for\\ the wave equation}
\author{Giuseppe Negro}
\thanks{Supported by the ERC grant 277778 and the MINECO grants SEV-2011-0087, SEV-2015-0554, MTM2013-41780-P and MTM2017-85934-C3-1-P (Spain), and by the LAGA, Universit\'e Paris 13 (France).}
\email{g.negro@bham.ac.uk}
\date{\today}
\begin{document}

\begin{abstract} We disprove a conjecture of Foschi, regarding extremizers for the Strichartz inequality with data in the Sobolev space $\Hdot^{1/2}\times\Hdot^{-1/2}(\R^d)$, for even $d\ge 2$. On the other hand, we provide evidence to support the conjecture in odd dimensions and refine his sharp inequality in $\R^{1+3}$, adding a term proportional to the distance of the initial data from the set of extremizers. The proofs use the conformal compactification of the Minkowski space-time given by the Penrose transform.

\vspace{1em}

\noindent{\bf Une in\'egalit\'e de Strichartz pr\'ecis\'ee pour l'\'equation des ondes.} \textsc{R\'esum\'e.} Nous infirmons une conjecture de Foschi concernant les points extr\'emaux de l'in\'egalit\'e de Strichartz \`a donn\'ees dans l'espace de Sobolev $\Hdot^{1/2}\times\Hdot^{-1/2}(\R^d)$, o\`u $d\ge 2$ est pair. En revanche, nous donnons des indications en faveur de sa conjecture en dimension impaire, ainsi qu'une version raffin\'ee de son in\'egalit\'e optimale sur $\mathbb R^{1+3}$, en ajoutant un terme proportionnel \`a la distance des donn\'ees initiales de l'ensemble des points extr\'emaux. Les d\'emonstrations utilisent la compactification conforme de l'espace-temps de Minkowski donn\'ee par la transformation de Penrose.
\end{abstract}

\maketitle

\section{Introduction}
We consider solutions $u$ to the wave equation $u_{tt}=\Delta u$, on $\R^{1+d}$ with $d\ge 2$, and initial data $\ubold(0)=(u(0), u_t(0))$. We take such initial data in the Sobolev space
of pairs $\fbold=(f_0,f_1)$ with norm defined by
\begin{equation}\label{eq:Sobolev_norm}
	\norm{\fbold}_{\Hdot^{1/2}\times\Hdot^{-1/2}(\R^d)}= \Big( \lVert\mhDelta^{1/4} f_0 \rVert_{L^2(\R^d)}^2 + \lVert\mhDelta^{-1/4} f_1\rVert_{L^2(\R^d)}\Big)^{1/2}.
\end{equation}
In 1977, Strichartz~\cite{Strichartz77} proved that there is a positive constant $C$ such that
\begin{equation}\label{eq:StrichartzIneq}
     \norm{u}_{L^{2\frac{d+1}{d-1}}(\R^{1+d})} \le C\norm{\ubold(0)}_{\Hdot^{1/2} \times \Hdot^{-1/2} (\R^d)}.
\end{equation}
Foschi \cite{Foschi07} proved that, for $d=3$, the optimal constant is attained when
\begin{equation}\label{eq:seeding_maximizer}
	 \ubold(0)= (u(0), u_t(0)) = \Tonde{ (1+\abs{\cdot}^2)^{-\frac{d-1}{2}} , 0}, 
\end{equation}
meaning that $\mathcal{A}_d=\norm{u}_{L^{2\frac{d+1}{d-1}}(\R^{1+d})}/\norm{\ubold(0)}_{\Hdot^{1/2} \times \Hdot^{-1/2}}$, where $\ubold(0)$ is given by~\eqref{eq:seeding_maximizer}, is the smallest possible value for the multiplicative constant $C$ in~\eqref{eq:StrichartzIneq} for $d=3$; precisely, $\StrichartzConstant=\Tonde{\frac 3 {16\pi} }^{1/4}$. Foschi conjectured that \eqref{eq:seeding_maximizer} should extremize in any dimension $d\ge 2$. We will provide evidence to support his conjecture in odd dimensions, however we will disprove it in even dimensions; see the forthcoming Theorem~\ref{put}.

Foschi also characterized the initial data that extremize the Strichartz inequality~\eqref{eq:StrichartzIneq} with $d=3$. The full set $\Mrom$ is obtained by acting a group of symmetries of the inequality on the data \eqref{eq:seeding_maximizer}.
Writing
\begin{equation}\label{eq:distance}
  \dist( \fbold, \Mrom) = \inf_{\phibold\in \Mrom}\norm{ \fbold - \phibold}_{\Hdot^{1/2}\times \Hdot^{-1/2}},
\end{equation}
we will mainly be concerned with the following refinement of Foschi's inequality.
\begin{thm}\label{thm:main} There is a positive constant $C$ such that, for all $u\colon\R^{1+3}\to \R$ satisfying $u_{tt}=\Delta u$,
\begin{equation*}
  	C\dist(\ubold(0), \Mrom)^2 \le \StrichartzConstant^2\norm{\ubold(0)}_{\Hdot^{1/2}\times \Hdot^{-1/2}}^2-\norm{u}_{L^4(\R^{1+3})}^2 \le\StrichartzConstant^2\dist(\ubold(0), \Mrom)^2.
\end{equation*} 
\end{thm}
The upper bound is proved in a more general setting in the following section. The lower bound, on the other hand, requires a much more careful treatment and it will follow from a local version, in which we also obtain the optimal constant.

Brezis and Lieb asked if the sharp Sobolev inequality due to Aubin \cite{Aub76} and Talenti \cite{Tal76} could be sharpened in this way; see \cite[question (c)]{BreLie85}. This was solved by Bianchi and Egnell \cite{BiaEgn91}. 
Most relevantly, a sharpening of the Strichartz inequality for the Schr\"odinger equation with $d=1$ or $2$ is implicit in the work of Duyckaerts, Merle and Roudenko \cite{DuMeRo11}, who applied their result to the mass-critical nonlinear Schr\"odinger equation in the small data regime (see also \cite{Gon17_2}). Theorem~\ref{thm:main} has a similar application to the cubic nonlinear wave equation in \cite{Giu18}.

 In the fourth section we consider the deficit functional $\psi$, defined as
 \begin{equation*}
 	\begin{array}{cc}
		\psi(\ubold(0)):=\mathcal{A}_d^{p}\norm{\ubold(0)}^{p}_{\Hdot^{1/2}\times\Hdot^{-1/2}}-\norm{u}_{L^p(\R^{1+d})}^{p},& p:=2\frac{d+1}{d-1},
	\end{array}
\end{equation*}
so that $\psi$ is zero at the supposed extremizers~\eqref{eq:seeding_maximizer}. Now for~\eqref{eq:seeding_maximizer} to be extremizing for the Strichartz inequality~\eqref{eq:StrichartzIneq}, it must be critical for $\psi$ in the sense that the first derivative of $\psi$ must also vanish there.  We will prove the following result disproving Foschi's conjecture in even dimensions.
 
\begin{thm}\label{put} The data \eqref{eq:seeding_maximizer} are  critical for $\psi$
if and only if $d\ge 2$ is odd.
 \end{thm}
The case of spatial dimension $d=2$ is especially surprising. Indeed, in the aforementioned~\cite[eq.~(46)]{Foschi07}, Foschi proved that $u_\pm(0)=(1+\lvert\cdot\rvert^2)^{-1/2}$, which is the same function that appears in~\eqref{eq:seeding_maximizer}, is an extremizer of the closely related half-wave estimate $\lVert u_\pm\rVert_{L^6(\mathbb R^{1+2})}\le (2\pi)^{-1/6}\lVert u_\pm(0) \rVert_{\dot{H}^{1/2}},$ in which $\partial_t u_\pm=\pm i\sqrt{-\Delta}u_\pm$; see also~\cite{BeJeOzSa17}. 
 
 In the fifth section, we prove the lower bound of Theorem~\ref{thm:main}. For this we must show that a spectral gap, associated with the second derivative of $\psi$, is positive. This is achieved using the Penrose transform, introduced in the third section. Under this transformation, the extremizing pair \eqref{eq:seeding_maximizer} 
is mapped to the constant initial data pair $(1/2,0)$, enabling explicit computations. 
	Compactness arguments will also be required to extend a local version of Theorem \ref{thm:main} to the whole space $\Hdot^{1/2}\times \Hdot^{-1/2}$. For this we will require a profile decomposition due to Ramos \cite{Ramos12}, also presented in the third section. 
 
We end the introduction with a mention of the recent paper~\cite{GoNe20}, in which sharp Strichartz estimates for the wave, the half-wave and the Schr\"odinger equations are studied by means of spacetime transformations such as the Penrose and the Lens transform.  
  
\section{Abstract upper bounds} \label{sec:upper_bounds}
In this section, $X$ will denote a measure space and $H$ will denote a real or complex Hilbert space, with scalar product $\langle, \rangle$ and norm $\lVert\cdot\rVert$.  
\begin{prop}\label{prop:up_bound} For $1<p<\infty$, let $S:H\to L^p(X)$ be a linear bounded
operator, with operator norm $\lVert S\rVert $. Then
\begin{equation}\label{eq:up_bound_main}
    \lVert S\rVert^2 \lVert f\rVert^2 - \lVert Sf\rVert_p^2 \le \lVert S\rVert^2 d(f,M_S)^2,
\end{equation}
where $M_S=\{v\in H:\lVert Sv\rVert_p=\lVert S\rVert \lVert v\rVert \}$ and $\displaystyle d(f, M_S)=\inf_{v\in M_S} \lVert f-v\rVert$.
\end{prop}
We remark that $M_S$ is never empty, as $0\in M_S$. Also, if $v\in M_S$ then $\lambda v\in M_S$ for all scalars $\lambda$; in particular, either $M_S=\{0\}$ or $M_S$ contains nonzero elements of arbitrarily small norm.
\begin{proof}[Proof of Proposition~\ref{prop:up_bound} ]
If $S=0$ or $M_S=\{0\}$, then~\eqref{eq:up_bound_main} is trivially true. Otherwise, we let $D=d(f,M_S)$, and for $\eps>0$ we consider a $v\in M_S$ such that $\lVert f-v\rVert^2\le
D^2+\eps$. By the previous remark, we can assume that $v\ne 0$, which implies that $Sv\ne 0$, because $\lVert Sv\rVert_p = \lVert S\rVert \lVert v\rVert\ne 0$. 

Now we write $g=f-v$ and we define, for $t\in\mathbb{R}$,
\begin{equation}\label{eq:h_one_two}
    \begin{array}{cc}
        h_1(t)=\lVert S(v+tg)\rVert_p^2-\lVert Sv\rVert_p^2, & 
h_2(t)=\lVert S(v+tg)\rVert_p^2-\lVert S\rVert^2\lVert v+tg\rVert^2.
    \end{array}
\end{equation}
As a function of $t$, each of $h_1,h_2$ is a difference of two convex functions and hence is left and right differentiable at every point. In addition, both are differentiable at $t=0$, since $Sv\neq 0$; see, for example,~\cite[Theorem 2.6]{LieLos01}

Now, $h_2$ has a maximum at zero, so $h_2'(0)=0$. Since $h_1$
is convex and $h_1(0)=0$, we have $h_1(1)\ge h_1'(0)$. Therefore
\begin{equation}\label{eq:big_proof_sec_two}
    \begin{split}
        h_1(1)=\lVert Sf\rVert_p^2-\lVert Sv\rVert_p^2 &\ge h_1'(0)=(h_1-h_2)'(0)= 2\lVert S\rVert^2\Re(\langle v,g\rangle) \\
         &=\lVert S\rVert^2(\lVert v+g\rVert^2 - \lVert g\rVert^2-\lVert v\rVert^2).
    \end{split}
\end{equation} 
Recalling that $v+g=f$ and $\lVert Sv\rVert_p=\lVert S\rVert \lVert v\rVert$, this yields
\begin{equation}\label{eq:end_proof_sec_two}
    \lVert Sf\rVert_p^2 \ge \lVert S\rVert^2(\lVert v+g\rVert^2-\lVert g\rVert^2)\ge \lVert S\rVert^2\lVert f\rVert^2 -(D^2+\eps)\lVert S\rVert^2,
\end{equation}
and since this holds for arbitrary $\eps>0$, the desired conclusion~\eqref{eq:up_bound_main} follows.
\end{proof}
The upper bound in Theorem~\ref{thm:main} is an immediate consequence of the latter proposition, obtained by specializing the operator $S$ to the wave propagator $S_t\colon \dot{\mathcal{H}}^{1/2}(\mathbb R^3) \to L^4(\mathbb R^{1+3})$; see the next section. 
\begin{rem}\label{rem:frank_follows}
	For $s\in (0, d)$, letting $S$ denote the fractional Sobolev embedding $\dot{H}^{s/2}(\mathbb R^d)\hookrightarrow L^{2d/(d-s)}(\mathbb R^d)$,  Proposition~\ref{prop:up_bound} gives an alternative proof of the upper bound of~\cite{CheFraWet12}.
\end{rem} 
\begin{rem}\label{rem:not_necessary_Lp}
    The only property of the target space $L^p(X)$ used in the proof of Proposition~\ref{prop:up_bound} is that its norm is Gateaux differentiable away from the origin. The Banach spaces with this property are called \emph{smooth}, and they admit various alternative characterizations; see~\cite[Section~5.4]{Meg98}.
\end{rem}

\section{Notation and preliminaries}
In the remainder of this paper, all functions will be real-valued unless otherwise stated. We use the following notation for the space of initial data: 
\begin{equation}\label{eq:PhaseSpaceNotation}
	\PhaseSpace(\R^d)= \Hdot^{1/2}(\R^d)\times \Hdot^{-1/2} (\R^d).
\end{equation}
Elements of $\PhaseSpace(\R^d)$, denoted with boldface, are considered as row or column vectors indifferently;
\begin{equation}\label{eq:boldface_notation}
	\fbold=(f_0, f_1)= \begin{bmatrix} f_0 \\ f_1\end{bmatrix} \in \PhaseSpace.
\end{equation}
The space $\PhaseSpace(\R^d)$ is a real Hilbert space, obtained by taking the completion of the Schwartz space with the scalar product 
\begin{equation}\label{eq:PhaseSpaceScalProd}
	\Braket{\fbold| \gbold}_{\PhaseSpace(\R^d)} = \int_{\R^d} \mhDelta^\sReg f_0 \cdot g_0\, dx  + \int_{\R^d} \mhDelta^{-\frac12}f_1 \cdot g_1\, dx,
\end{equation}
with a standard abuse of notation in the second integral, since $f_1, g_1$ are just distributions. The symbol $\bot$ will reflect orthogonality with respect to this scalar product.

We denote by $\fboldstar$ the following element of $\PhaseSpace(\R^d)$: 
\begin{equation}\label{eq:extremizer_seed}
	\fboldstar= \Tonde{ 2^{\frac{d-1}{2}}(1+\abs{\cdot}^2)^{-\frac{d-1}{2}} , 0 },
\end{equation}
which is an extremizer of the Strichartz inequality when $d=3$. As we mentioned in the introduction, the Strichartz inequality is invariant under the action of a Lie group of symmetries, which we now represent on $\PhaseSpace$. Most of the following definitions and computations will be needed only in the case $d=3$, but there is no added difficulty in considering the general case $d\ge 2$. For $t\in\R$, the symbols $S_t$ and $\Svec_t$ will denote the wave propagators, defined by
\begin{equation}\label{eq:wave_propagator}
	S_t\fbold = \cos(t\mhDelta^\frac12)f_0  + \frac{\sin(t\mhDelta^\frac12)}{\mhDelta^\frac12} f_1
\end{equation}
and 
\begin{equation}\label{eq:wave_propagator_vector}
	\Svec_t\fbold=\begin{bmatrix} \cos(t\mhDelta^\frac12) & \frac{\sin(t\mhDelta^\frac12)}{\mhDelta^\frac12}  \\ 
		-\sin(t\mhDelta^\frac12)\mhDelta^\frac12 & \cos(t\mhDelta^\frac12) \end{bmatrix} 
		\begin{bmatrix} f_0 \\ f_1\end{bmatrix} .
\end{equation}
For $\theta\in\SSS^1:=\R/2\pi\Z$, the symbol $\Ph_{\theta}$ will denote phase shift; 
\begin{equation}\label{eq:PhaseSymmetry}
	\Ph_\theta \fbold = \PhaseMatrix{\theta} \begin{bmatrix} f_0 \\ f_1 \end{bmatrix},
\end{equation}
which is characterized by
\begin{equation}\label{eq:PhaseCharacterization}
	\Svec_t\Ph_\theta \fbold = \Ph_\theta\Svec_t \fbold = \begin{bmatrix} \Ds \cos(t\sqrtDelta +\theta ) & \Ds \frac{\sin(t\sqrtDelta +\theta)}{\sqrtDelta }  \\ 
		\Ds -\sin(t\sqrtDelta +\theta )\sqrtDelta  &\Ds \cos(t\sqrtDelta +\theta) \end{bmatrix} 
		\begin{bmatrix} f_0 \\ f_1\end{bmatrix},
\end{equation}
For $\zeta_j\in\R$ and $j=1,\ldots, d$, the symbol $L_{\zeta_j}^j$ will denote the Lorentz boost along the $x_j$ axis, given by 
\begin{equation}\label{eq:Lorentz_boost}
		L_{\zeta_j}^j \fbold =(\left.u_{\zeta_j}\right|_{t=0}, \left.\partial_t u_{\zeta_j}\right|_{t=0}).
\end{equation}
Here $u(t, x)=S_t\fbold$ and 
\begin{align}\label{eq:Lorentz_uboost}
	u_{\zeta_1}(t, x)&= u(t \cosh \zeta_1  + x_1 \sinh\zeta_1 , t\sinh \zeta_1  + x_1 \cosh\zeta_1 ,x_2,\ldots , x_d) ,\\
	\notag 
	u_{\zeta_2}(t, x)&= u(t \cosh \zeta_2 + x_2 \sinh\zeta_2, x_1, t\sinh \zeta_2  + x_2\cosh\zeta_2 ,\ldots, x_d) ,\\
	\notag   & \vdotswithin{=} \\
	\notag u_{\zeta_d}(t, x)& =  u(t \cosh \zeta_d + x_d \sinh\zeta_d ,x_1, \ldots,x_{d-1}, t \sinh \zeta_d  + x_d \cosh\zeta_d).
\end{align} 
We introduce the collective parameter $\alphabold\in\SSS^1\times \R^{2d+2}\times \SO(d)$;
\begin{equation}\label{eq:alpha_collective_index}\begin{array}{cc}
	\alphabold=
		(\theta, t_0, \zeta_1,\ldots ,\zeta_d, \sigma, x_0, A), &  \theta\in \SSS^1, t_0,\zeta_j,\sigma\in \R, x_0\in \R^d, A\in \SO(d).
		\end{array}
\end{equation}
Then for $\fbold\in\Hcaldot^\frac12$ we define the linear operator
\begin{equation}\label{eq:cGamma_notation}
	\Gamma_{\!\alphabold}\fbold\!=\!
			\Svec_{t_0}\Ph_\theta L_{\zeta_1}^1\ldots L_{\zeta_d}^d
			 \Big( e^{\frac{d-1}{2}\sigma} f_0\Tonde{e^\sigma A(\cdot + x_0)}, e^{\frac{d+1}{2}\sigma} f_1\Tonde{e^\sigma A(\cdot + x_0)} \Big),
\end{equation}
and we write, conventionally, $\Gamma_{\!\obold}$ to denote the identity operator. We have the invariances 
\begin{equation}\label{eq:Gamma_alphabold_is_unitary}
	\begin{array}{ccc}
	\norm{\Gamma_{\!\alphabold} \fbold}_{\PhaseSpace} =\norm{\fbold}_{\PhaseSpace}&\text{and} & \norm{S_t\Gamma_{\!\alphabold} \fbold}_{L^4(\R^{1+d})}=\norm{S_t\fbold}_{L^4(\R^{1+d})}.
	\end{array}
\end{equation}
We will prove these after introducing some more notation.
\begin{rem}\label{rem:UnitGrp}	
    The full action of the symmetry group on the Strichartz inequality is the transformation $\fbold\mapsto c\Gamma_{\!\alphabold} \fbold$. This notation has been chosen to highlight the difference between the multiplicative transformation $\fbold\mapsto c \fbold$, which is a symmetry of the inequality but does not satisfy \eqref{eq:Gamma_alphabold_is_unitary}, and the transformation $\Gamma_{\!\alphabold}$, which preserves both sides of the inequality. 
    
    In this regard, we also mention that the second identity in~\eqref{eq:Gamma_alphabold_is_unitary} is specific of the space $L^4(\R^{1+d})$, as the operator $\Ph_\theta$ does not seem to preserve the $L^p(\R^{1+d})$ norm unless $p=4$. For $d\ne 3$, such a  $L^4$ norm may be infinite for some $\fbold\in \dot{\mathcal{H}}^{1/2}$; in this case, both sides of the second identity in~\eqref{eq:Gamma_alphabold_is_unitary} are infinite.
\end{rem}
As shown by Foschi \cite{Foschi07},  the set $\Mrom$ of extremizers of the sharp three-dimensional Strichartz inequality coincides with the orbit of $\fboldstar$;
\begin{equation}\label{eq:Mrom_Manifold}
	\begin{split}
	\Mrom &= \Set{  \fbold\in\Hcaldot^\frac12(\R^3) | \norm{S_t\fbold}_{L^4(\R^{1+3})} 
	= \StrichartzConstant \norm{\fbold}_{\Hcaldot^\frac12}} \\
	&=  \Set{ c \Gamma_{\!\alphabold} \fboldstar | c\in\R , \alphabold\in
		 \SSS^1 \times\R^{8}\times \SO(3)}.
	\end{split}
\end{equation}
\begin{rem}\label{rem:complex_maximizers}
    We choose to work with real-valued functions only, since it is a natural assumption in the context of the wave equation, also from the point of view of optimal constants and extremizers. For example, we mention that, as it is proved in~\cite[Proposition~3.1]{GoNe20}, the real or the imaginary part of a complex extremizer must be a real extremizer. Note, however, that in~\cite{Foschi07} complex solutions are considered, so \eqref{eq:PhaseSymmetry} is replaced by two independent symmetries and the set of extremizers $\Mrom$ is larger in this case. 
\end{rem}

The set $\Mrom\setminus\Set{\obold}$ has the structure of a finite-dimensional differentiable manifold. The tangent space to $\Mrom$ at $\obold\ne \fbold\in \Mrom$ is
\begin{equation}\label{eq:tangent_space_Mrom}
	T_\fbold \Mrom = \Span \Set{ \fbold, \left.\nabla_{\!\alphabold} \Gamma_{\!\alphabold} \fbold\right|_{\alphabold=\boldsymbol{0}}}, 
\end{equation}
where $\nabla_{\!\alphabold}$ is the list of derivatives with respect to the parameters \eqref{eq:alpha_collective_index}. We refer to such derivatives as the \emph{generators} of the symmetry group. We will give the explicit expression of the generators in Table~\ref{tab:generator_table}.
\begin{rem}\label{rem:rotations}
    The generators associated to the parameter $A\in \SO(d)$ are the differential operators $x_i\partial_{x_j}-x_j\partial_{x_i}$, for $1\le i<j\le d$. However, since we will always work in the orbit of $\fboldstar$, which is radially symmetric, these generators will play no role in this paper.
\end{rem}
We will give an explicit description of the tangent space at $\fboldstar$; this suffices to describe the tangent space at all $\fbold\in\Mrom\setminus\{\obold\}$, as the following proposition shows.
\begin{prop}\label{prop:TangentSpaces}
For all $c\ne 0$, 
\begin{equation}\label{eq:pushforward_tangent_spaces}
	T_{c\Gamma_{\!\alphabold} \fboldstar} \Mrom=\Gamma_{\!\alphabold}\Tonde{ T_{\fboldstar}\Mrom}.
\end{equation}
\end{prop}
\begin{proof}
By definition, 
\begin{equation}\label{eq:def_tansp_pushforward}
	T_{c\Gamma_{\!\alphabold} \fboldstar} \Mrom=\Span\Set {c\Gamma_{\!\alphabold}\fboldstar, c \left.\nabla_{\!\betabold}(\Gamma_{\!\betabold} \Gamma_{\!\alphabold} \fboldstar)\right\rvert_{\betabold=\obold} }.
\end{equation}
By basic Lie theory, the map $\Gamma_{\!\betabold}\mapsto (\Gamma_{\!\alphabold})^{-1}\Gamma_{\!\betabold}\Gamma_{\!\alphabold}$ is a Lie group homomorphism. Thus, there is a differentiable function $\gammabold=\gammabold(\betabold)$, with $\gammabold(\obold)=\obold$, such that $(\Gamma_{\!\alphabold})^{-1}\Gamma_{\!\betabold}\Gamma_{\!\alphabold}=\Gamma_{\!\gammabold(\betabold)}$. By the chain rule,
 \begin{equation}\label{eq:proof_tansp_pushfo_two}
	\left.\partial_{\beta_j}(\Gamma_{\!\betabold} \Gamma_{\!\alphabold} \fboldstar)\right\rvert_{\betabold=\obold} = \Gamma_{\!\alphabold}(\left.\partial_{\beta_j}\Gamma_{\!\gammabold(\betabold)} \fboldstar\right\rvert_{\betabold=\obold})=\Gamma_{\!\alphabold}\sum_{k}c_{kj}\left.\partial_{\gamma_k}\Gamma_{\!\gammabold}\fboldstar\right|_{\gammabold=\obold},
\end{equation}
	where $c_{kj}:=\frac{\partial \gamma_k}{\partial \beta_j}(\obold)$. This proves that $T_{c\Gamma_{\!\alphabold} \fboldstar}\Mrom \subset \Gamma_{\!\alphabold}(T_\fboldstar \Mrom)$. The reverse inclusion is proven in the same way.
\end{proof}

\begin{table}[!]
\begin{equation}\label{eq:generator_table}
	\begin{array}{c|c|c}
		&\text{Derivative} & \text{Applied to }c\Gamma_{\!\alphabold} \fbold \text{ at }c=1, \alphabold=\obold\\ 
		\toprule
	1&	\frac{\partial}{\partial c} & \fbold\\
	2&	\frac{\partial}{\partial t_0}&  \begin{bmatrix} 0 & 1 \\ \Delta & 0\end{bmatrix}\fbold\\ 
	3&	\frac{\partial}{\partial \theta}& \begin{bmatrix} 0 & \mhDelta^{-\frac12} \\ -\mhDelta^\frac12  & 0\end{bmatrix} \fbold\\
	4&	\frac{\partial}{\partial \zeta_j}& \begin{bmatrix} 0& x_j \\ x_j\Delta + \frac{\partial}{\partial x_j} & 0 \end{bmatrix} \fbold \quad (j=1,2,\ldots, d)\\
	5&	\frac{\partial}{\partial \sigma}& \begin{bmatrix} \frac{d-1}{2}+ x\cdot \nabla & 0\\ 0 & \frac{d+1}{2} + x\cdot \nabla\end{bmatrix} \fbold \\
	6&	\nabla_{\!x_0}& \begin{bmatrix} \frac{\partial}{\partial x_j} & 0 \\  0 & \frac{\partial}{\partial x_j} \end{bmatrix}\fbold\quad (j=1,2,\ldots, d).
		\end{array}
\end{equation}
\caption{Symmetry generators.}
\label{tab:generator_table}
\end{table}
\begin{proof}[Proof of \eqref{eq:Gamma_alphabold_is_unitary}]
	The proof of the first identity in~\eqref{eq:Gamma_alphabold_is_unitary} reduces to a check that the operators in the right column of entries 2-6 of Table~\ref{tab:generator_table} are skew-adjoint on $\Hcaldot^\frac12(\R^d)$. We remark that this is true for any dimension $d$. The second identity in~\eqref{eq:Gamma_alphabold_is_unitary}, concerning invariance of the $L^4(\R^{1+d})$ norm, is obvious for all symmetries except for $\Ph_\theta$ (defined in \eqref{eq:PhaseSymmetry}). This invariance is proved in~\cite[equation~(2.5)]{BezRogers13} in the case $d=5$, but the proof applies verbatim to arbitrary $d\ge 2$.
\end{proof}

We now cast in our notation the profile decomposition result of Ramos~\cite[Theorem~3.1]{Ramos12} which extends the profile decomposition~\cite{BaGe99} of Bahouri and G\'erard (see also Merle and Vega~\cite{MeVe98} for the Schr\"odinger equation) to the regularity $\PhaseSpace$ and includes the Lorentz symmetry. 

\begin{thm}\label{thm:profile_decomposition}
 Let $\fbold_n$ be a bounded sequence in $\PhaseSpace(\R^3)$. Then there exists a finite or infinite sequence $\Set{\fbold^j\ :\ j=1, 2\ldots}\subset \PhaseSpace$ and corresponding sequences of transformations $\Gamma_{\!\alphabold_n^j}$ (defined in \eqref{eq:cGamma_notation}) such that, up to passing to a subsequence, 
 \begin{equation}\label{eq:profile_decomposition}
 	\fbold_n=\sum_{j=1}^J \Gamma_{\!\alphabold_n^j} \fbold^j + \rbold^J_n,
\end{equation}
where the remainder term $\rbold^J_n$ satisfies 
\begin{equation}\label{eq:profile_decomposition_smallness_remainder_term}
	\lim_{J\to\infty} \limsup_{n\to \infty} \norm{S_t\rbold^J_n}_{\Lfour} =0.
\end{equation}
Moreover, for each $J\ge 1$ the following Pythagorean expansions hold for $n\to \infty$:
\begin{equation}\label{eq:Pythagorean_energy}
\Ds\norm{\fbold_n}_\PhaseSpace^2 = \sum_{j=1}^J \norm{ \fbold^j}_\PhaseSpace^2 + \norm{\rbold^J_n}_\PhaseSpace^2 + o(1), 
\end{equation}
and 
\begin{equation}\label{eq:Pythagorean_strichartz}
	\norm{S_t\fbold_n}_\Lfour^4 = \sum_{j=1}^J \norm{ S_t \fbold^j }_\Lfour^4 + \norm{S_t \rbold^J_n}_\Lfour^4 + o(1). 
\end{equation} 
\end{thm}

We introduce now the Penrose transform; see \cite{Pen64}. We do this for general spatial dimension $d\ge 2$. The Penrose transform is a map $\Pcal$ of  $\R^{1+d}$ onto a bounded region $\Pcal (\R^{1+d})$ of the Lorentzian manifold $\R\times \SSS^d$. Adopting the notation of H\"ormander \cite[Appendix A.4]{Hor97} we parameterize 
\begin{equation}\label{eq:sphere_definition}
	\SSS^d = \Set{ X=(X_0, \Xvec) \in \R\times \R^{d} : X_0^2+X_1^2+\ldots + X_d^2=1}
\end{equation}
by the polar coordinates 
\begin{equation}\label{eq:sphere_polar_coordinates}
	\begin{array}{cc}
	X=(X_0, \Xvec)=(\cos(\RPolar) , \sin(\RPolar)\, \omega) , & \RPolar\in [0, \pi],\ \omega \in \SSS^{d-1} .
	\end{array}
\end{equation}
The Penrose transform is the map
\begin{equation}\label{eq:Penrose_map_arrows}
	\begin{array}{rcl}
		\Pcal\colon \R^{1+d}&\to& \R\times \SSS^d \\ 
		(t, r\omega) & \mapsto & (T, \cos \RPolar, \sin \RPolar\,\omega),
	\end{array}
\end{equation}
where $r=\abs{x}, \omega = \frac{x}{\abs{x}}$ and 
\begin{equation}\label{eq:Penrose_map_equations}
	\begin{array}{c}
		T=\arctan(t+r)+\arctan(t-r),    \\ 
		R=\arctan(t+r)-\arctan(t-r).
	\end{array}
\end{equation}
The image $\Pcal(\R^{1+d})$ is the region 
\begin{equation}\label{eq:image_of_penrose}
	\Pcal(\R^{1+d}) = \Set{ \Big(T, (\cos R, \sin R\,\omega)\Big)\in\R\times \SSS^d | \begin{array}{c} 
		-\pi<T<\pi \\  
		0\le R\le \pi-\abs{T} \\  
		\omega\in \SSS^{d-1}
		\end{array}
		}.
\end{equation}
The map $\Pcal$ is conformal in the sense that, applying the change of coordinates \eqref{eq:Penrose_map_equations}, one has 
\begin{equation}\label{eq:conformal_transformation}
	dT^2-dX^2 = \Omega^2\Tonde{ dt^2 - dx^2},
\end{equation}
where the conformal factor $\Omega$ is given by
\begin{equation}\label{eq:penrose_conformal_factor}
		\Omega = 2(1+(t+r)^2)^{-\frac12}(1+(t-r)^2)^{-\frac12} = \cos T + \cos R.
\end{equation}
The restriction of the Penrose transform to the initial time slice $\Set{t=0}$ is the stereographic projection from the south pole of $\SSS^d$:
\begin{equation}\label{eq:Pzero}
	\Pzero:=\left.\Pcal\right|_{t=0} \colon  \Set{t=0}\times \R^d \to \Set{T=0}\times \Tonde{ \SSS^d\setminus\Set{(-1,0, \ldots ,0)}}.
\end{equation}	 
This is also a conformal map, whose conformal factor we denote
\begin{equation}\label{eq:stereographic_conformal_factor}
		\OmegaZero=\left.\Omega\right|_{t=0}=2(1+r^2)^{-1} = 1+\cos R.
\end{equation}

We now introduce spherical harmonics. We use the notation $Y_{\ell,m}$ for normalized real-valued spherical harmonics on $\SSS^d$. Here $\ell\in \mathbb N_{\ge 0}$ denotes the degree and $m$ the degeneracy. We thus have
\begin{equation}\label{eq:spherical_harmonics}
	\begin{array}{cc}
		-\Delta_{\SSS^d} Y_{\ell,m} = \ell(\ell+d-1)Y_{\ell,m}, & m=0, \ldots, \Degeneracy(\ell):=\frac{(2\ell+d-1)(\ell+d-2)!}{\ell!(d-1)!}	-1,
		\end{array}
\end{equation}
and 
\begin{equation}\label{eq:spherical_harmonics_normalization}
	\int_{\SSS^d} Y_{\ell,m}(X)^2\, dS = 1,
\end{equation}
where $dS$ is the surface measure on $\SSS^d$. We recall that $Y_{\ell,m}(X)$ is the restriction to $\SSS^d$ of a homogeneous harmonic polynomial of degree $\ell$ in the variables $X=(X_0, X_1, \ldots, X_d)\in \R^{d+1}$. In particular, 
\begin{equation}\label{eq:spherical_harmonics_symmetry}
	Y_{\ell,m}(-X)=(-1)^\ell Y_{\ell,m}(X).
\end{equation}
For each $\ell\ge 0$, there is exactly one \emph{zonal} normalized spherical harmonic (up to an irrelevant sign); that is, one that is a function of $X_0$ only (see, for example,~\cite[§2, Lemma~1]{Muller98}). We denote it by $Y_{\ell, 0}$. For example, the spherical harmonics of degree $0$ and $1$ are
\begin{equation}\label{eq:low_deg_sph_harm}
	\begin{array}{cc}
		Y_{0,0}=\frac{1}{\sqrt{\abs{\SSS^d}}}, & Y_{1, m}(X)=\sqrt{\frac{d+1}{\abs{\SSS^d}}}X_m,\quad (m=0,1,\ldots, d).
	\end{array}
\end{equation}

We use the hat notation to denote the coefficients of expansions in spherical harmonics: if $F\in L^2(\SSS^d)$, we write
\begin{equation}\label{eq:SphHarm_Expansion}
	F(X)=\sum_{\ell=0}^\infty \sum_{m=0}^{\Degeneracy(\ell)} \Fhat(\ell, m) Y_{\ell,m}(X),
\end{equation}
We will use the fractional operators $A_1$ and $A_{-1}$ on $\SSS^d$, defined by their action on spherical harmonics:
\begin{equation}\label{eq:intertwining_ops}
	A_{\pm 1}Y_{\ell,m}:=\Tonde{-\Delta_{\SSS^d} +\Tonde{\frac{d-1}{2}}^2}^{\pm\frac12}Y_{\ell,m} = \Tonde{\ell+\frac{d-1}{2}}^{\pm 1} Y_{\ell,m}.
\end{equation}
These operators are the lifting to $\SSS^d$ of the euclidean fractional Laplacians $\mhDelta^{\pm \frac12}$ via the stereographic projection $\Pzero$, in the sense that, for any scalar field $F$ on $\SSS^d$:
\begin{equation}\label{eq:morpurgo}
	(A_{\pm 1}F)\circ \Pzero = \OmegaZero^{-\frac12(d\pm 1)}\Tonde{-\Delta}^{\pm\frac12}\Tonde{\OmegaZero^{\frac12(d\mp 1)} F \circ\Pzero};
\end{equation}
see \cite[equation (2)]{Mor02}.

The conformality of the Penrose transform $\Pcal$ implies that the substitutions 
\begin{equation}\label{eq:Penrose_field_transformation}
	\begin{array}{rcl}
		\Omega^{\frac{1-d}{2}} u &=& U\circ \Pcal \\ 
		\OmegaZero^{\frac{1-d}{2}} f_0 &=& F_0\circ\Pzero \\ 
		\OmegaZero^{\frac{-1-d}{2}}f_1 &=& F_1\circ \Pzero
	\end{array}
\end{equation}
have the property that 
\begin{equation}\label{eq:u_solves_flat_wave}
	\begin{array}{ccc}
	\begin{cases}
		u_{tt}=\Delta u,\ \ \text{on }\R^{1+d} \\ 
		\left.u\right|_{t=0} = f_0 \\ 
		\left.u_t\right|_{t=0} = f_1 \\ 
	\end{cases} 
	&\!\!\!\!\!\!\!\!\iff\!\!\!\!
	&
	\begin{cases}
	U_{TT} =\Delta_{\SSS^d} U - \Tonde{\frac{d-1}{2}}^2 U,\ \ \text{on }\Pcal(\R^{1+d}) \\ 
	\left.U\right|_{T=0}=F_0 \\ 
	\left.U_T\right|_{T=0}=F_1.
	\end{cases}
	\end{array}
\end{equation}
The expansion of $U$ in spherical harmonics reads 
\begin{equation}\label{eq:U_sph_harm}
	\begin{split}
	U(T, X)=\sum_{\ell=0}^\infty\sum_{m=0}^{\Degeneracy(\ell)}&\cos\Tonde{T(\ell+\frac{d-1}{2})}\Fhat_0(\ell,m) Y_{\ell,m}(X)  \\  
	& +\frac{\sin\Tonde{T\Tonde{\ell+\frac{d-1}{2}}}}{\ell+\frac12(d-1)}\Fhat_1(\ell,m)Y_{\ell,m}(X).
	\end{split}
\end{equation}
Actually, this formula defines a function on $\R\times \SSS^d$. The restriction of this function to $\Pcal(\R^{1+d})$ corresponds to the solution $u$ of the wave equation on $\R^{1+d}$. If $d$ is odd, $U$ is $2\pi$-periodic in $T$ and it satisfies 
\begin{equation}\label{eq:USymOddDim}
	\begin{array}{cc}
		U(T+\pi, -X)=(-1)^\frac{d-1}{2}U(T, X), & \forall (T, X)\in \SSS^1\times \SSS^d,
	\end{array}
\end{equation}
because of the sign property \eqref{eq:spherical_harmonics_symmetry} of $Y_{\ell,m}$. If $d$ is even, \eqref{eq:USymOddDim} fails.

The Strichartz extremizer \eqref{eq:extremizer_seed} can be written as follows: 
\begin{equation}\label{eq:extremizer_seed_conformal_notation}
	\fboldstar=\Tonde{\OmegaZero^{\frac{d-1}{2}}, 0}.
\end{equation}
Therefore, if $\ustar=S_t\fboldstar$ and $\Ustar$ are related by \eqref{eq:Penrose_field_transformation}, with corresponding initial data $\fboldstar$ and $(F_{\star\, 0}, F_{\star\,1})$, then we have the particularly simple expressions
\begin{equation}\label{eq:Ustar_is_cosine}
	\begin{array}{ccc}
	F_{\star\,0}=1,& F_{\star\,1}=0, &	\Ustar(T, X)=\cos \Tonde{\frac{d-1}{2}T}.
	\end{array}
\end{equation}
To facilitate forthcoming computations, we remark that $\Fhat_{\star\,0}(0,0)=\sqrt{\abs{\SSS^d}}$ and $\Fhat_{\star\,0}(\ell, m)=0$ for $\ell\ge1$. 

We now discuss integration. Letting $dS$ denote the surface measure on $\SSS^d$, if $F\colon \SSS^d\to \R$ and $V\colon \Pcal(\R^{1+d})\to \R$ one has the following change of variable formulas: 
\begin{align}\label{eq:varchange_stereographic}
		\int_{\R^d} F(\Pzero(x))\, dx &= \int_{\SSS^d} F(X)\OmegaZero^{-d}\,dS(X) \\
		\label{eq:varchange_Penrose}
		\iint_{\Pcal(\R^{1+d})} V(T, X)\, dTdS(X) & = \iint_{\R^{1+d}} V(\Pcal(t, x))\Omega^{d+1}\, dtdx
\end{align}
It is a consequence of the first formula and of equation \eqref{eq:morpurgo} that, if $\fbold, \gbold$ are related to $(F_0, F_1)$ and $(G_0, G_1)$ via \eqref{eq:Penrose_field_transformation}, then 
\begin{equation}\label{eq:scalprod_Aone}
	\Braket{\fbold|\gbold}_{\Hcaldot^\frac12}=\int_{\SSS^d} A_1 F_0 \cdot G_0\, dS +\int_{\SSS^d} A_{-1} F_1 \cdot G_1\, dS,
\end{equation}
and so
\begin{equation}\label{eq:scalprod_stereographic}
	\begin{split}
	\Braket{\fbold|\gbold}_{\Hcaldot^\frac12}=\sum_{\ell=0}^\infty\sum_{m=0}^{\Degeneracy(\ell)} &\Tonde{\ell+\frac{d-1}{2}}\Fhat_0(\ell, m)\Ghat_0(\ell, m)\\ &+\Tonde{\ell+\frac{d-1}{2}}^{-1}\Fhat_1(\ell, m)\Ghat_1(\ell, m).
	\end{split}
\end{equation}
In particular, from \eqref{eq:Ustar_is_cosine} it follows that
\begin{equation}\label{eq:fboldstar_norm}
	\norm{\fboldstar}_{\Hcaldot^\frac12(\R^d)}^2 = \frac{d-1}{2}\lvert \SSS^d\rvert.
\end{equation}
Using the symmetry \eqref{eq:USymOddDim} we can considerably simplify spacetime integrals. 
\begin{lem} 
\label{lem:sym_trick}
If $V$ is a function on $\mathbb S^1 \times \mathbb S^d$ that satisfies
\begin{equation}\label{eq:VSym}
	\begin{array}{cc}
		V(T+\pi, -X)=V(T, X), & \forall (T, X)\in \SSS^1\times \SSS^d,
	\end{array}
\end{equation}
then 
\begin{equation}\label{eq:symmetry_trick_conclusion}
	\iint_{\Pcal(\R^{1+d})} V(T, X)\, dT dS(X) = \frac12 \iint_{\SSS^1\times \SSS^d} V(T, X)\, dTdS(X).
\end{equation}
\end{lem}
\begin{proof} 
	We use the polar coordinates \eqref{eq:sphere_polar_coordinates}, so that 
	\begin{equation}\label{eq:sphere_metric_tensor}
		dS = (\sin R)^{d-1}\, dR\, dS^{d-1}
	\end{equation}
	where $dS^{d-1}$ denotes the volume element on $\SSS^{d-1}$; see \cite[§1.42]{Muller98}. Setting 
	\begin{equation}\label{eq:G_integral}
		G(R)=\int_{-\pi+R}^{\pi-R}\Tonde{ \int_{\SSS^{d-1}}V(T, \cos R, \sin R\, \omega)\,dS^{d-1}(\omega)}\,dT, 
	\end{equation}
	the integral to evaluate can be rewritten as 
	\begin{equation}\label{eq:symmetry_trick_step_one}
		\iint_{\Pcal(\R^{1+d})} V(T, X) \, dT dS(X) = \int_0^\pi (\sin R)^{d-1}\frac12\Tonde{ G(R) + G(\pi-R)}\, dR
	\end{equation}
	Using the changes of variable $\omega\mapsto -\omega$ and $T\mapsto T \pm \pi$, 
	\begin{equation}\label{eq:G_pi_minus_R}
		\begin{split}
		G(\pi-R)=&\int_{-\pi}^{-\pi+R}\Tonde{ \int_{\SSS^{d-1}}V(T-\pi, -\cos R, -\sin R\, \omega) \, dS^{d-1}}\, dT  \\ 
		& +\int_{\pi-R}^\pi\Tonde{ \int_{\SSS^{d-1}} V(T+\pi, -\cos R, -\sin R\, \omega)\,dS^{d-1}}\, dT.
		\end{split}
	\end{equation}
	Inserting~\eqref{eq:G_pi_minus_R} into~\eqref{eq:symmetry_trick_step_one} and using the assumption \eqref{eq:USymOddDim}, we obtain \eqref{eq:symmetry_trick_conclusion}.
\end{proof}
\begin{cor}\label{cor:sym_trick}
	Let $d$ be an odd integer. If $u_{tt}=\Delta u$ and $w_{tt}=\Delta w$ on $\R^{1+d}$, and if $U, W$ are related to $u, v$ via the Penrose transform \eqref{eq:Penrose_field_transformation}, then
	\begin{equation}\label{eq:wave_sym_trick}
		\iint_{\R^{1+d}} \abs{u}^a \abs{w}^b\, dtdx = \frac12 \iint_{\SSS^1\times \SSS^d} |\Omega|^{\frac{d-1}{2}(a+b)-(d+1)} |U|^a|W|^b\, dTdS,
	\end{equation}
	and
	\begin{equation}\label{eq:wave_sym_trick_alt}
		\iint_{\R^{1+d}} \abs{u}^{a-1} u\, w\, dtdx = \frac12 \iint_{\SSS^1\times \SSS^d} |\Omega|^{\frac{d-1}{2}(a+1)-(d+1)} |U|^{a-1}U\, W\, dTdS,
	\end{equation}
	 for all $a,b\in\R$. Here $\Omega(T, X)=\cos T + X_0$, where $X=(X_0, \Xvec)\in \SSS^d$.
\end{cor}
\begin{proof} To prove \eqref{eq:wave_sym_trick}, we need to check that 
\begin{equation}\label{eq:V_sym_trick_proof}
	V(T, X)= |\Omega|^{\frac{d-1}{2}(a+b)-(d+1)} |U|^a|W|^b
\end{equation}
satisfies the property \eqref{eq:VSym}, which is an immediate consequence of the symmetry property \eqref{eq:USymOddDim} of $U$ and $W$. We remark that these symmetry properties need not hold for even $d$. The proof of \eqref{eq:wave_sym_trick_alt} is analogous.
\end{proof}

We end the section with the  computation of the tangent space $T_{\fboldstar}\Mrom$, defined in \eqref{eq:tangent_space_Mrom}, where we recall that $\fboldstar$ is the extremizer given in \eqref{eq:extremizer_seed_conformal_notation}. We systematically use the following identification of $x\in \R^d$ with $X\in \SSS^d$ via the stereographic projection $\Pzero$:
\begin{equation}\label{eq:stereographic_projection}
	\begin{array}{ccc}
		 \OmegaZero-1=X_0, & x_j\OmegaZero = X_j, & j=1,\ldots, d.
	\end{array}
\end{equation}
In the following equations, the first computation is performed by applying \eqref{eq:stereographic_projection}, the second by applying \eqref{eq:morpurgo} once, and the last by applying \eqref{eq:morpurgo} twice:
\begin{align}\notag 
   \frac \partial{\partial x_j}\Tonde{ \OmegaZero^\frac{d-1}{2} } &= -\frac{d-1}{2}x_j\OmegaZero^{\frac{d+1}{2}}=-\frac{d-1}{2}X_j\OmegaZero^{\frac{d-1}{2}},\\
    \label{eq:OmegaFormulasGeneralDim}\mhDelta^\frac12\OmegaZero^{\frac{d-1}{2}} &= \frac{d-1}{2}\OmegaZero^{\frac{d+1}{2}} ,\\ 
    \notag-\Delta \OmegaZero^{\frac{d-1}{2}} &= \frac{d-1}{2}\OmegaZero^{\frac{d+1}{2}}\Tonde{\frac{d-1}{2} + \frac{d+1}{2}X_0}   .
 \end{align}
From \eqref{eq:OmegaFormulasGeneralDim} and \eqref{eq:stereographic_projection}, using $\sum_{j=1}^d X_j^2=1-X_0^2$ we infer 
\begin{equation}\label{eq:OmegaFormulasCont}
	x\cdot \nabla \Tonde{ \OmegaZero^\frac{d-1}{2}} = -\frac{d-1}{2}(1-X_0^2)\OmegaZero^{\frac{d-3}{2}} = -\frac{d-1}{2}(1-X_0)\OmegaZero^\frac{d-1}{2}.
\end{equation}
We apply the generators of the symmetry group, listed in Table~\ref{tab:generator_table}, to the Strichartz extremizer $\fboldstar$ given in \eqref{eq:extremizer_seed_conformal_notation}. Using the computations \eqref{eq:OmegaFormulasGeneralDim}, we obtain Table~\ref{tab:symmetry_table}; we recall that we are identifying $x\in \R^d$ with $X\in \SSS^d$ via the stereographic projection \eqref{eq:stereographic_projection}.
\begin{table}[!]
\begin{equation}\label{eq:symmetry_table}
	\begin{array}{c|c|c} 
		&\text{Generator} & \text{Applied to }\fboldstar=\Tonde{\OmegaZero^{\frac{d-1}{2}}, 0} \\
		\toprule \\ 
		1& \begin{bmatrix} 1 &0 \\0 & 1 \end{bmatrix}  & \begin{bmatrix}\OmegaZero^{\frac{d-1}{2}} \\ 0 \end{bmatrix}  \\ 
	2 & \begin{bmatrix}0&  1 \\ \Delta & 0\end{bmatrix}& \begin{bmatrix} 0 \\ -\frac{d-1}{2}\OmegaZero^{\frac{d+1}{2}}\Tonde{ \frac{d-1}{2} + \frac{d+1}{2}X_0}  \end{bmatrix}\\ 
	3 & \begin{bmatrix} 0& \mhDelta^{-\frac12} \\ -\mhDelta^\frac12 &0 \end{bmatrix}& \begin{bmatrix} 0 \\ -\frac{d-1}2 \OmegaZero^{\frac{d+1}{2}}\end{bmatrix} \\ 
	4& \begin{bmatrix} 0& x_j \\ x_j \Delta + \frac{\partial}{\partial{x_j}} & 0\end{bmatrix} & \begin{bmatrix} 0 \\ -\frac{d-1}{2} \OmegaZero^{\frac{d-1}{2}}\Tonde{ \frac{(d-1)(d+1)}{4} + \frac{d+1}{2}X_0 }X_j  \end{bmatrix} \\
		5& \begin{bmatrix} \frac{d-1}{2}+ x\cdot \nabla &0 \\ 0& \frac{d+1}{2}+x\cdot \nabla \end{bmatrix} & \begin{bmatrix}  \frac{d-1}{2}X_0\OmegaZero^{\frac{d-1}{2}} \\ 0 \end{bmatrix} \\ 
	6&\begin{bmatrix} \frac{\partial}{\partial x_j}&0 \\ 0& \frac{\partial}{\partial x_j} \end{bmatrix} & \begin{bmatrix} -\frac{d-1}{2} X_j \OmegaZero^{\frac{d-1}{2}} \\ 0\end{bmatrix}\quad (j=1\ldots d).
	\\ 
	\end{array}
\end{equation}
\caption{A basis of the tangent space at $\fboldstar$ in arbitrary dimension.}
\label{tab:symmetry_table}
\end{table}
Since $\OmegaZero=1+X_0$, when $d=3$ the fourth line of Table~\ref{tab:symmetry_table} simplifies:  
\begin{equation}\label{eq:OmegaThreeMiracle}
	\OmegaZero^{\frac{d-1}{2}}\Tonde{ \frac{(d-1)(d+1)}{4}+ \frac{d+1}{2}X_0 }X_j=2\OmegaZero^2X_j.
\end{equation} 
So, specializing Table~\ref{tab:symmetry_table} to the case $d=3$, we conclude that 
\begin{equation}\label{eq:tangent_space_dim_three}
	T_{\fboldstar}\Mrom = \Set{ \begin{bmatrix} \OmegaZero P(X) \\ \OmegaZero^2 Q(X)\end{bmatrix} : P, Q\text{ polynomials of degree}\le 1\text{ in }X\in \SSS^3}.
\end{equation}
Since the restrictions of these polynomials to the sphere are spherical harmonics of degree $0$ and $1$, after applying the Penrose transform \eqref{eq:Penrose_field_transformation} we see that 
\begin{equation}\label{eq:belong_to_tangent}
	\begin{array}{ccl}
		\fbold\in T_{\fboldstar} \Mrom & \iff & \Fhat_0(\ell, m)=\Fhat_1(\ell, m)=0, \quad \ell\ge2.
	\end{array}
\end{equation}
 In light of the identity \eqref{eq:scalprod_stereographic}, expressing the $\PhaseSpace(\R^3)$ scalar product in terms of $F_0, F_1$, we characterize the orthogonal complement of $T_\fboldstar \Mrom$ as follows:
\begin{equation}\label{eq:ortho_to_tangent}
	\begin{array}{ccl}
		\fbold\bot T_{\fboldstar}\Mrom & \iff & \Fhat_0(\ell,m)=\Fhat_1(\ell,m)=0,\quad \ell=0, 1.
	\end{array}
\end{equation}

\section{Proof of Theorem~\ref{put}}
Here we consider the functional 
\begin{equation}\label{eq:DefFunct}
	\begin{array}{cc}
		\psi(\fbold):=\mathcal{A}_d^p \norm{\fbold}_\PhaseSpace^p - \norm{S_t \fbold}_{L^p(\R^{1+d})}^p,& p:=2\frac{d+1}{d-1},
	\end{array}
\end{equation}
where $\mathcal{A}_d=\norm{S_t\fboldstar}_{L^p(\R^{1+d})}/\norm{\fboldstar}_\PhaseSpace$ and $\fboldstar$ is the pair of initial data defined in~\eqref{eq:extremizer_seed}. Note that both $\mathcal{A}_d$ and $\fboldstar$ depend on the dimension $d\ge 2$. Theorem~\ref{put} can be recast as follows. 
\begin{thm}\label{thm:put}
	It holds that 
	\begin{equation}\label{eq:put}
		\begin{array}{cc}
			\left.\frac{d}{d\eps}\psi(\fboldstar + \eps \fbold)\right|_{\eps=0} = 0 , & \forall\fbold \in \PhaseSpace(\R^d),
		\end{array}
	\end{equation}
	if and only if $d$ is odd.
\end{thm}
\begin{lem}\label{lem:Penrose_first_order}
	Writing $\fbold=c\fboldstar + \fboldbot$, with $\Braket{\fboldbot  | \fboldstar}_\PhaseSpace= 0$, then 
	\begin{equation}\label{eq:Penrose_first_order}
		\left.\frac{d}{d\eps}\psi(\fboldstar+\eps\fbold)\right|_{\eps=0}=-p\iint_{\R^{1+d}}\abs{S_t \fboldstar}^{p-2}S_t\fboldstar S_t\fboldbot\, dtdx.
	\end{equation}
\end{lem}
\begin{proof}
	This follows from the computation
	\begin{equation}\label{eq:Penrose_first_order_comp}
		\left.\frac{d}{d\eps}\psi(\fboldstar+\eps\fbold)\right|_{\eps=0}\!\!\! = p\mathcal{A}_d^p\Braket{\fboldstar|\fbold}_\PhaseSpace \norm{\fboldstar}_\PhaseSpace^{p-2}-p\iint_{\R^{1+d}}\abs{S_t\fboldstar}^{p-2}S_t\fboldstar S_t \fbold\, dtdx,
	\end{equation}
	which holds for any $\fbold\in\PhaseSpace(\R^d)$, and then taking  $\fbold=c\fboldstar + \fboldbot$ and recalling the definition of $\mathcal{A}_d$.
	\end{proof}
When $d$ is odd, using Corollary \ref{cor:sym_trick} and \eqref{eq:Ustar_is_cosine} we can rewrite the integral on the right-hand side of \eqref{eq:Penrose_first_order} as follows: 
\begin{equation}\label{eq:Penrose_first_order_odd}
	\iint_{\R^{1+d}}\abs{S_t \fboldstar}^{p-2}S_t\fboldstar S_t\fboldbot\, dtdx = \frac12 \iint_{\SSS^1\times \SSS^d} \abs{\cos \frac{d-1}{2}T}^{p-2}\!\!\!\!\cos\left(\frac{d-1}{2}T\right) U_\bot\, dTdS,
\end{equation}
where $u_\bot=S_t\fboldbot$ and $U_\bot$ are related by the Penrose transform \eqref{eq:Penrose_field_transformation}. From the formula \eqref{eq:scalprod_stereographic} we infer that the condition $\Braket{\fboldstar|\fboldbot}_\PhaseSpace\!=\!0$ is equivalent to $\Fhat_{\bot\, 0}(0,0)=0$. Therefore, expanding $U_\bot$ in spherical harmonics as in \eqref{eq:U_sph_harm}, we see that $U_\bot(T, \cdot)$ satisfies
\begin{equation}\label{eq:Ubot_ortho}
	\begin{array}{cc}\Ds
	\int_{\SSS^d} U_\bot(T, X)\, dS(X)=C\sin\Tonde{\frac{d-1}{2}T} \Fhat_1(0,0), & \forall\,T\in [-\pi, \pi],
	\end{array}
\end{equation}
 for some constant $C$. This implies that
 \begin{multline}\label{eq:Penrose_first_order_odd_wip}
 	\frac12 \iint_{\SSS^1\times \SSS^d} \abs{\cos \frac{d-1}{2}T}^{p-2}\cos\left(\frac{d-1}{2}T\right) U_\bot\, dTdS \\ = \frac{C}{2}\Fhat_1(0,0)\int_{-\pi}^\pi\abs{\cos \frac{d-1}{2}T}^{p-2}\cos\left(\frac{d-1}{2}T\right)\sin\left(\frac{d-1}{2}T\right)\, dT=0,
\end{multline}
as the last integrand is odd. This completes the proof of Theorem \ref{thm:put} in the odd dimensional case.

The reason why this argument fails in even dimension is that Corollary~\ref{cor:sym_trick} is not applicable in that case. In order to prove that, in fact, $\fboldstar$ is not a critical point in even dimension, we need only prove that the derivative is nonzero in a single direction. A bad choice would be to take the direction $\fbold=(f_0, 0)$, where $f_0$ corresponds to a spherical harmonic of degree $1$ under the Penrose transform \eqref{eq:Penrose_field_transformation}, as then we would be moving in the direction of the symmetries of the inequality; see entries 5 and 6 in Table~\ref{tab:symmetry_table}. Instead we consider a zonal spherical harmonic of degree $2$, which we denote by $Y_{2, 0}$ in agreement with the notation of Section 3.
\begin{lem}\label{lem:critpoint_failure}
	Let $d\ge 2$ be even and let $\fbold=(f_0, 0)\in \PhaseSpace(\R^d)$ be the initial data corresponding to 
	\begin{equation}\label{eq:Second_SphHarm}
		\begin{array}{cc}
			F_0= Y_{2, 0}, & F_1=0,
		\end{array}
	\end{equation}
	via the Penrose transform \eqref{eq:Penrose_field_transformation}. Then 
	\begin{equation}\label{eq:critpoint_failure}
		\begin{array}{ccc}
		\Ds \left.\frac{d}{d\eps}\psi(\fboldstar+\eps\fbold)\right|_{\eps=0}=(-1)^{\frac{d}{2}+1}	\, c_d, &\text{where} & c_d>0.
		\end{array}
	\end{equation}
\end{lem}
\begin{proof}
Applying the Penrose transform to \eqref{eq:Penrose_first_order} we obtain 
\begin{equation}\label{eq:critpoint_failure_wip}
	\left.\frac{d}{d\eps} \psi(\fboldstar + \eps \fbold)\right|_{\eps=0} = -p\iint_{\Pcal(\R^{1+d})} \abs{\cos \frac{d-1}{2}T}^{p-2}\cos\Tonde{\frac{d-1}{2}T} U\, dTdS,
\end{equation}
where $U(T, X_0, \vec X)=\cos\Tonde{(2+\frac{d-1}{2})T}Y_{2, 0}(X_0)$; see \eqref{eq:U_sph_harm}. As in the previous section, we have written the generic point $X\in \SSS^d$ as $X=(X_0, \Xvec)$, where $X_0\in[-1, 1]$, to exploit the fact that $Y_{2, 0}$ is a function of $X_0$ only. Taking into account the definition \eqref{eq:image_of_penrose} of $\Pcal(\R^{1+d})$, the right-hand side of the previous identity reads 
\begin{multline}\label{eq:critpoint_failure_wipII}
	-p\lvert \SSS^{d-1}\rvert\int_{-\pi}^\pi \abs{\cos \frac{d-1}{2}T}^{p-2}\cos\Tonde{\frac{d-1}{2}T} \cos\Tonde{(2+\frac{d-1}{2})T}\, dT \\ 
	\times \int_0^{\pi-\abs{T}} Y_{2, 0}(\cos R)\,(\sin R)^{d-1}\, dR.
\end{multline}
We have used the formula $dS=(\sin R)^{d-1}dR dS^{d-1}$ for the volume element of $\SSS^d$ in the polar coordinates \eqref{eq:sphere_polar_coordinates}. The zonal spherical harmonic $Y_{2, 0}$ can be expressed by the Rodrigues formula: 
\begin{equation}\label{eq:Rodrigues}
	Y_{2, 0}(X_0)= R_{2, d}(1-X_0^2)^{-\frac{d-2}{2}}\frac{d^2}{dX_0^2} (1-X_0^2)^{2+\frac{d-2}{2}}; 
\end{equation}
 see~\cite[Lemma 4, pg.\,22]{Muller98}, where $R_{2, d}>0$ is a constant whose exact value is not important here. We compute the last integral in \eqref{eq:critpoint_failure_wipII} using the change of variable $X_0=\cos R$:
\begin{equation}\label{eq:last_integral}
	\begin{split}
	\int_0^{\pi-\abs{T}} Y_{2, 0}(\cos R)(\sin R)^{d-1}\, dR&= R_{2, d}\int_{-\cos T}^1 \frac{d^2}{dX_0^2} (1-X_0^2)^{2+\frac{d-2}{2}}\, dX_0 \\ &= C_d \cos T (\sin T)^d,
	\end{split}
\end{equation}
where $C_d>0$. Inserting this into \eqref{eq:critpoint_failure_wipII} shows that it remains to prove the following:
\begin{equation}\label{eq:Remains_to_Prove}
	\begin{array}{cc}
	\Ds I(d):=\frac1{\pi}\int_{-\pi}^\pi h_d(T) P_d(T)\, dT = (-1)^\frac d 2 c_d, & \text{for some }c_d>0,
	\end{array}
\end{equation}
where $h_d(T):=\abs{ \cos \frac{d-1}{2} T}^{p-2}$ and
\begin{equation}\label{eq:actors}
		P_d(T):=\cos\Tonde{\frac{d-1}{2}T} \cos\Tonde{\frac{d+3}{2}T}\cos T (\sin T)^d.
\end{equation}

We first consider the case $d=2$. In this case we have that $p=6$, so we can evaluate $I(2)$ explicitly: 
\begin{equation}\label{eq:Itwo_evaluation}
	\begin{split}
	I(2)&=\frac1\pi\int_{-\pi}^\pi \Tonde{\cos \frac T 2}^5\cos\frac{5T}{2}\cos T\Tonde{\sin T}^2\, dT\\ 
	&=\frac4\pi\int_0^{\pi/2}(\cos T)^5\cos 5T\cos 2T(\sin 2T)^2\, dT=-\frac{5}{128}.
	\end{split}
\end{equation}
In the case $d\ge 4$ we will use the Parseval identity:
\begin{equation}\label{eq:Parseval_aftermath}
	I(d)=\frac{\hat{h}_d(0)\hat{P}_d(0)}2+\sum_{k=1}^\infty \hat{h}_d(k)\hat{P}_d(k),
\end{equation}
where $\hat{f}(k):=\frac{1}{\pi}\int_{-\pi}^\pi f(T)\cos(kT)\, dT$. We remark that, with this choice of notation, 
\begin{equation}\label{eq:normalization_Fourier}
\text{if }f(T)=\frac{a_0}2  +\sum_{k=1}^\infty a_k \cos(kT), \text{ then } a_k=\hat{f}(k).
\end{equation} 
\begin{lem}\label{lem:FractFourierSeries}
	If $k\ne m(d-1)$ where $m\in\N_{\ge 0}$ then $\hat{h}_d(k)=0$.
\end{lem}
\begin{proof}[Proof of Lemma \ref{lem:FractFourierSeries}.] 
	Consider $u\in[-\sqrt{2}, \sqrt{2}]$. We set $\abs{u}^{p-2}=(1+v)^\frac{p-2}{2}$, with $v=u^2-1$, and we expand it using the binomial series. This yields
	\begin{equation}\label{eq:u_expansion}
		\Ds |u|^{p-2}=\sum_{j=0}^\infty \binom{(p-2)/2}{j}(u^2-1)^j,
	\end{equation}
	and the series converges uniformly by Raabe's criterion (here we use that $p>2$). Taking $u=\cos\frac{d-1}{2}T$, we obtain
	\begin{equation}\label{eq:cos_expansion}
		\abs{\cos \frac{d-1}{2}T}^{p-2}=\sum_{j=0}^\infty (-1)^j \binom{(p-2)/2}{j}\Tonde{\sin \frac{d-1}{2}T}^{2j},
	\end{equation}
	For each $j\in \N_{\ge 0}$ we can develop
	\begin{equation}\label{eq:SinFract}
		\begin{split}
		\Tonde{ \sin \frac{d-1}{2} T}^{2j}&= \frac{(-1)^j}{2^{2j}}\Tonde{ e^{i\frac{d-1}{2}T}-e^{-i\frac{d-1}2 T}}^{2j} \\		
&=\frac{(-1)^j}{2^{2j}}\sum_{m=0}^{2j}\binom{2j}{m}(-1)^me^{i(j-m)(d-1)T} \\
&=\frac{1}{2^{2j}}\Tonde{ \binom{2j}{j} + 2\sum_{m=1}^j \binom{2j}{j-m}(-1)^m \cos (m(d-1)T)}.
	\end{split}
	\end{equation}
	This shows that each summand in \eqref{eq:cos_expansion} is a linear combination of the terms $\cos(m(d-1)T)$, with $m\in\N_{\ge 0}$, which in light of \eqref{eq:normalization_Fourier} completes the proof.
 \end{proof}	
We now turn to the term $P_d$ introduced in \eqref{eq:actors}. Using the addition formula for the cosine, and developing $(\sin T)^d$ like we did in the previous proof, we can express $P_d$ as a trigonometric polynomial of degree $2(d+1)$:
\begin{equation}\label{eq:Pd_first_exp}
	\begin{split}
		P_d(T)= &2^{-d-2}\Tonde{\cos T + \cos 3T + \cos dT + \cos (d+2)T} \\ 
		&\times \Tonde{\binom{d}{d/2} + 2\sum_{k=1}^{d/2} (-1)^k\binom{d}{d/2-k}\cos(2kT)};
	\end{split}
\end{equation}
so, in particular, $\hat{P}_d(k)=0$ if $k> 2(d+1)$. Since $d\ge 4$, we infer from this and from Lemma \ref{lem:FractFourierSeries} that $I(d)$ reduces to the sum of four terms: 
\begin{equation}\label{eq:FractFourierAftermath}
	\begin{split}
	I(d)=&\frac{1}{2}\hat{h}_d(0)\hat{P}_d(0)+ \sum_{m=1}^{3}\hat{h}_d(m(d-1))\hat{P}_d(m(d-1)).
	\end{split}
\end{equation}
Actually, we have that $\hat{P}_d(3(d-1))=0$. This is obvious for $d\ge 6$, because in that case $3(d-1)$ exceeds $2(d+1)$, and can be established for $d=4$ by inspection of the formula 
\begin{equation}\label{eq:Pfour_explicit}
	\begin{split}
		P_4(T)= &2^{-6}\Tonde{\cos T + \cos 3T + \cos 4T + \cos 6T}\Tonde{6-8\cos 2T +2\cos 4T},
	\end{split}
\end{equation}
again using \eqref{eq:normalization_Fourier}. 

To compute the remaining coefficients, we use the addition formula for the cosine to rewrite \eqref{eq:Pd_first_exp} as
\begin{equation}\label{eq:Pd_second_exp}
	2^{d+2}P_d(T)=P_{d, 1}(T)+P_{d,3}(T)+P_{d, d}(T)+P_{d, d+2}(T) , 
\end{equation}
where each summand is given by
\begin{equation}\label{eq:Pdh}
	P_{d, h}(T)=\binom{d}{d/2}\cos hT + \sum_{k=1}^{d/2}(-1)^k\binom{d}{d/2-k}(\cos(2k-h)T + \cos(2k+h)T),
\end{equation}
for $h=1,3, d, d+2$. To compute $\hat{P}_d(0)$, we observe that the only contributing term is obtained for $2k-h=0$, and that can only happen for $h=d$ and $k=d/2$. By \eqref{eq:normalization_Fourier} we have
\begin{equation}\label{eq:hatPzero}
	2^{d+2}\hat{P}_d(0)=\hat{P}_{d,d}(0)=2(-1)^\frac d 2.
\end{equation}
To compute $\hat{P}_d(d-1)$ we observe that, as $d-1$ is odd, the only contributing terms are obtained for $h=1, 3$: 
\begin{equation}\label{eq:hatPdminusone}
	\begin{split}
	2^{d+2}\hat{P}_d(d-1)&=\hat{P}_{d,1}(d-1)+\hat{P}_{d, 3}(d-1)\\ 
		&=(-1)^\frac d 2 -(-1)^\frac d 2\binom{d}{1}+ (-1)^\frac d 2\binom{d}{2}\\ 
		&=(-1)^\frac d 2 \frac{(d-1)(d-2)}{2}.
	\end{split}
\end{equation}
With analogous reasoning we obtain
\begin{equation}\label{eq:hatPtwodminusone}
	\begin{split}
		2^{d+2}\hat{P}_d(2(d-1)) &=  \hat{P}_{d,d}(2(d-1))+\hat{P}_{d, d+2}(2(d-1)) \\
					&=-(-1)^{\frac d 2} \binom{d}{1} +(-1)^{\frac{d}{2}}\binom{d}{2}\\
					&=(-1)^\frac d 2 \Tonde{\frac{(d-1)(d-2)}{2}-1}.
	\end{split}
\end{equation}
Inserting the preceding computations into Parseval's identity \eqref{eq:FractFourierAftermath}, we obtain the formula 
\begin{multline}\label{eq:Id_Fourier}
	(-1)^\frac{d}{2}2^{d+2}I(d)=\hat{h}_d(0) -\hat{h}_d(2(d-1)) \\+\frac{(d-1)(d-2)}{2}\Tonde{\hat{h}_d(d-1)+\hat{h}_d(2(d-1))}.
\end{multline}

To conclude the proof of \eqref{eq:Remains_to_Prove}it will suffice to prove that
\begin{equation}\label{eq:sign_relations}
	\begin{array}{ccc}
		\hat{h}_d(0)- \hat{h}_d(2(d-1))>0, & \text{and} & \hat{h}_d(d-1)+ \hat{h}_d(2(d-1)) >0.
	\end{array}
\end{equation}
The first inequality follows immediately from the definition \eqref{eq:actors} of $h_d$:
\begin{equation}\label{eq:FirstIneq}
	 \hat{h}_d(0)-\hat{h}_d(2(d-1)) =\frac1\pi \int_{-\pi}^\pi \abs{\cos \frac{d-1}{2}T}^{p-2}(1-\cos 2(d-1) T)\, dT>0.
\end{equation}
To prove the second inequality we note that the change of variable $T\mapsto \frac 2 {d-1} T$ produces 
\begin{equation}\label{eq:SecondIneqWIP}
	\hat{h}_d(d-1)+\hat{h}_d(2(d-1))=\frac2{\pi(d-1)}\int_{-\frac{d-1}{2}\pi}^{\frac{d-1}{2}\pi}\abs{\cos T}^{p-2}(\cos 2T + \cos 4T)\, dT.
\end{equation}
The integrand function in the right-hand side is $\pi$-periodic and even. Therefore, the integral is an integer multiple of the integral over $[0, \pi/2]$. Moreover, $\cos 2T+\cos 4T=2\cos T \cos 3T$. We get 
\begin{equation}\label{eq:SecondIneqWWIP}
	\hat{h}_d(d-1)+\hat{h}_d(2(d-1))=\frac{4(d-2)}{\pi(d-1)}\int_0^{\pi /2}\abs{\cos T}^{p-2}\cos T\cos 3T\, dT.
\end{equation}
To conclude the proof, we notice that 
\begin{equation*}
	\int_0^{\pi/6} \cos T\cos 3T\, dT = -\int_{\pi/6}^{\pi/2}\cos T\cos 3T\, dT>0,
\end{equation*}
and $\abs{\cos T}^{p-2}$ is strictly decreasing on $[0, \pi/2]$, so 
\begin{equation}\label{eq:EstimatesToSum}
	\int_0^{\pi/6} \abs{\cos T}^{p-2}\cos T\cos 3T\, dT > -\int_{\pi/6}^{\pi/2}\abs{\cos T}^{p-2}\cos T\cos 3T\, dT,
\end{equation}
which proves that the right-hand side in \eqref{eq:SecondIneqWWIP} is strictly positive. This shows that the second inequality in \eqref{eq:sign_relations} holds, and the proof of Theorem \ref{thm:put} is complete.
\end{proof}

	\section{Proof of the lower bound in Theorem\ \ref{thm:main}} \label{sec:three_d_proof}
 In this section the spatial dimension $d$ will be $3$, so that $p=4$ in the definition \eqref{eq:DefFunct} of the deficit functional $\psi$. We will use Corollary \ref{cor:sym_trick} to compute integrals on $\R^{1+3}$ using the Penrose transform, taking advantage of the simple expression \eqref{eq:Ustar_is_cosine} of $S_t\fboldstar$ under such transform. In particular, Foschi's constant $\StrichartzConstant=(3/16\pi)^\frac14$ has the representation 
 \begin{equation}\label{eq:Foschi_represented}
 	\StrichartzConstant^4=\frac{\norm{S_t \fboldstar }_{L^4(\R^{1+3})}^4}{\norm{\fboldstar}_{\PhaseSpace(\R^3)}^4}=\frac{\int_{-\pi}^\pi (\cos T)^4\, dT}{2\abs{\SSS^3}}.
\end{equation}
Here we have used the fact that $\norm{\fboldstar}_{\PhaseSpace(\R^3)}^2= \lvert \SSS^3 \rvert$; see \eqref{eq:fboldstar_norm}. 
\begin{lem}\label{lem:main}
There exists a quadratic functional $Q\colon \PhaseSpace(\R^3)\to [0, \infty)$ such that
	\begin{equation}\label{eq:four_homogeneous_Taylor_expansion_new}
	\psi(\fboldstar + \fbold)= Q(\fbold) + O(\norm{\fbold}^3_\PhaseSpace ),
\end{equation}
for all $\fbold\in \PhaseSpace(\R^3)$. It holds that $Q(\fbold)=0$ if and only if $\fbold\in T_\fboldstar\Mrom$, and moreover 
\begin{equation}\label{eq:Q_sharp}
	\begin{array}{cc}
		Q(\fbold)\ge \frac{\pi}{4}\norm{\fbold}_{\PhaseSpace}^2 , & \forall \fbold\bot T_{\fboldstar}\Mrom,
	\end{array}
\end{equation}
where the constant $\frac{\pi}{4}$ cannot be replaced by a larger one.
\end{lem}
\begin{proof}
We have that $\psi(\fboldstar)=0$ by definition of $\psi$, and we have proved in Theorem \ref{thm:put} that $\left.\frac{d}{d\eps} \psi(\fboldstar + \eps \fbold)\right|_{\eps=0}=0$ for all $\fbold\in \PhaseSpace(\R^3)$. So \eqref{eq:four_homogeneous_Taylor_expansion_new} holds with $Q(\fbold)=\frac12\left.\frac{d^2}{d\eps^2}\psi(\fboldstar+\eps \fbold)\right|_{\eps=0}$. Expanding we see that
\begin{equation}\label{eq:Q_functional}
	\begin{split}
	 Q(\fbold) = & 
	\,\StrichartzConstant^4\Tonde{ 
		4\Braket{ \fboldstar | \fbold}_\PhaseSpace^2 +2\norm{\fboldstar}_{\Hcaldot^\frac12}^2\norm{\fbold}_{\Hcaldot^\frac12}^2}-6 \iint_{\R^{1+\SpaceDim}} (S_t\fboldstar)^2(S_t\fbold)^2\, dtdx.
	\end{split}
\end{equation}
We record that, for all $\fbold=(f_0, f_1)\in \PhaseSpace(\R^3)$, it holds that 
\begin{equation}\label{eq:Qsymmetry}
	Q(\fbold)=Q(f_0, 0)+Q(0, f_1). 
\end{equation}
To prove this, we start by recalling that $\norm{\fbold}_\PhaseSpace^2=\Braket{f_0|f_0}_{\Hdot^\frac12}+\Braket{f_1|f_1}_{\Hdot^{-\frac12}}$. Moreover, since $\fboldstar=(f_{\star\, 0}, 0)$, we have that
	$\Braket{\fboldstar | \fbold}_\PhaseSpace = \Braket{f_{\star\, 0}|f_0}_{\Hdot^\frac12}$, so the first summand in the right-hand side of \eqref{eq:Q_functional} splits into the sum of a term depending on $f_0$ only and a term depending on $f_1$ only. The other summand splits in the same way; indeed, by the definition \eqref{eq:wave_propagator} of the wave propagator, $S_t\fboldstar = \cos(t\sqrtDelta)f_{\star\, 0}$, therefore
\begin{multline}\label{eq:SimplifTwo}
	\iint_{\R^{1+3}}\! (S_t\fboldstar)^2(S_t\fbold)^2 \!\! = \!\! 
	\iint_{\R^{1+3}}\!(S_t\fboldstar)^2(\cos t\sqrtDelta f_0)^2\! +\!\iint_{\R^{1+3}}\!(S_t\fboldstar)^2\!\Tonde{\frac{\sin t\sqrtDelta }{\sqrtDelta} f_1}^2  
	\\ \!\!\!+ 2\!\!\iint_{\R^{1+3}}\!(\cos t\sqrtDelta f_{\star\,0})^2\cos t\sqrtDelta f_0 \frac{\sin t \sqrtDelta}{\sqrtDelta}f_1,
\end{multline}
where the last integral vanishes, as can be seen with the change of variable $t\mapsto -t$. This proves \eqref{eq:Qsymmetry}. 

We now bound $Q(\fbold)$ from below, starting with the term $Q(f_0,0)$. We assume that $\fbold$ and $(F_0, F_1)$  are related via the Penrose transform \eqref{eq:Penrose_field_transformation}. By the formula \eqref{eq:scalprod_stereographic}, that expresses the $\PhaseSpace$ scalar product in terms of $(F_0, F_1)$, we rewrite the first summand in the right-hand side of \eqref{eq:Q_functional} as 
\begin{multline}\label{eq:Qfzero_pos}
	\StrichartzConstant^4\Tonde{4\Braket{f_{\star\,0} | f_0}_{\Hdot^\frac12}^2 + 2 \norm{f_{\star\,0}}_{\Hdot^\frac12}^2 \norm{ f_0 }_{\Hdot^\frac12}^2}=\\
	\frac{\int_{-\pi}^\pi (\cos T)^4\, dT}{2\abs{\SSS^3}} \Tonde{ 4 \abs{\SSS^3} \Fhat_0(0,0)^2 +2\abs{\SSS^3}\sum_{\ell=0}^\infty\sum_{m=0}^{\Degeneracy(\ell)}(\ell+1)\Fhat_0(\ell, m)^2},
\end{multline}
where we have used the property that $F_{\star\,0}=1=\sqrt{\abs{\SSS^3}}Y_{0,0}$; see \eqref{eq:Ustar_is_cosine}. We compute the other summand using Corollary \ref{cor:sym_trick}:
\begin{multline}
\label{eq:Qfzero_neg} 
	6\!\!\iint_{\R^{1+3}}\!\!(S_t\fboldstar)^2(\cos(t\sqrtDelta) f_0)^2 = 3\!\!\! \iint_{\SSS^1\times \SSS^3}\!\!\!\!\Tonde{\!\! \cos T \!\sum_{\ell, m}\cos(\ell+1)T\, \Fhat_0(\ell,m) Y_{\ell,m}\!}^2\!\!\!.
\end{multline}
By the $L^2(\SSS^3)$-orthonormality of $Y_{\ell,m}$, the right-hand side equals
\begin{multline}\label{eq:RHS_Q_functional}
	3\int_{-\pi}^\pi (\cos T)^4\,dT \Fhat_0(0,0)^2
	+3\sum_{\ell=1}^\infty\sum_{m=0}^{\Degeneracy(\ell)} \int_{-\pi}^\pi (\cos T\cos(\ell+1)T)^2\, dT \Fhat_0(\ell, m)^2.
\end{multline}
For all $\ell\ge1$, it holds that  
\begin{equation}\label{eq:trig_computation} 
	3\int_{-\pi}^\pi (\cos T\cos (\ell+1)T)^2\,dT=\frac{3\pi}{2}=2\int_{-\pi}^\pi (\cos T)^4\,dT,
\end{equation} 
so, subtracting the last equation from \eqref{eq:Qfzero_pos}, the terms corresponding to $\ell=0$ and $\ell=1$ vanish, and we obtain that
\begin{equation}\label{eq:Qfzero}
	Q(f_0, 0)=\frac{3\pi}4\sum_{\ell=2}^\infty \sum_{m=0}^{\Degeneracy(\ell)} (\ell-1)\Fhat_0(\ell, m)^2.
\end{equation}
The term $Q(0, f_1)$ is computed in the same way, and the end result is:
\begin{equation}\label{eq:Q_g_explicit}
	Q(\fbold)=\frac{3\pi}{4}\sum_{\ell=2}^\infty\sum_{m=0}^{\Degeneracy(\ell)} (\ell-1)\Quadre{\Fhat_0(\ell,m)^2 +\frac{\Fhat_1(\ell,m)^2}{(\ell+1)^2}}.
\end{equation}
From this we see that $Q(\fbold)=0$ if and only if $\Fhat_0(\ell, m)=\Fhat_1(\ell, m)=0$ for $\ell\ge 2$, which is equivalent to $\fbold\in T_\fboldstar \Mrom$; see \eqref{eq:belong_to_tangent}. 

It remains to prove the sharp inequality \eqref{eq:Q_sharp}. For $\ell\ge 2$, it holds that 
\begin{equation}\label{eq:TowSharpThree}
	\begin{array}{ccc}
		3(\ell-1)\ge \ell+1, &\text{and}& 3\frac{\ell -1}{(\ell+1)^2}\ge \frac{1}{\ell+1}, 
	\end{array}
\end{equation}
with equality for $\ell=2$. Therefore, \eqref{eq:Q_g_explicit} implies the sharp inequality
\begin{equation}\label{eq:sharp_Qg_ineq}
	Q(\fbold)\ge \frac{\pi}{4} \sum_{\ell=2}^\infty \sum_{m=0}^{\Degeneracy(\ell)}(\ell+1)\Fhat_0(\ell,m)^2+(\ell+1)^{-1}\Fhat_1(\ell,m)^2.
\end{equation}
The expression on the right-hand side equals $\frac\pi4 \norm{\fbold}_{\PhaseSpace(\R^3)}^2$ precisely when $\Fhat_0(\ell,m)=\Fhat_1(\ell,m)=0$ for $\ell=0,1$, which is equivalent to $\fbold\bot T_{\fboldstar}\Mrom$; see~\eqref{eq:ortho_to_tangent}. This completes the proof.
\end{proof}
\begin{rem}\label{rem:quad_form_symmetry}
	The fact that $Q(\fbold)=0$ for $\fbold\in T_\fboldstar \Mrom$ is a consequence of the criticality of $\fboldstar$ and of the invariance of $\psi$ under the symmetries $\Gamma_{\!\alphabold}$ (defined in~\eqref{eq:cGamma_notation}); indeed, differentiating the identity $\psi(c\Gamma_{\!\alphabold} \fboldstar)=0$ twice with respect to $c$ we get $Q(\fboldstar)=0$, and differentiating twice with respect to $\alpha_j$, we get 
\begin{equation*}
	Q\Tonde{\left.\frac{\partial}{\partial\alpha_j} \Gamma_{\!\alphabold}\fboldstar\right|_{\alphabold=\obold}}=0.
\end{equation*} 
	In Lemma \ref{lem:main} we proved a sharper result; namely, that $Q(\fbold)$ vanishes \emph{if and only if} $\fbold\in T_\fboldstar\Mrom$, and we gave a sharp explicit bound. In the language of the calculus of variations we can say that $\fboldstar$ is a \emph{non-degenerate} local minimizer of the deficit functional $\psi$, up to symmetries.  
\end{rem}

\begin{prop}\label{prop:main}
  For all $\fbold \in \Hcaldot^\frac12(\R^3)$ such that 
  \begin{equation}\label{eq:distance_condition}
  	\dist(\fbold, \Mrom)< \norm{\fbold}_{\Hcaldot^\frac12},
	\end{equation}
	it holds that
  \begin{equation}\label{eq:prop_main}
  	\frac{1}{3}\StrichartzConstant^2\dist(\fbold, \Mrom)^2 + O(\dist(\fbold, \Mrom)^3)
 \le \StrichartzConstant^2\norm{\fbold}_{\Hcaldot^\frac12}^2-\norm{S_t \fbold}_{L^4(\R^{1+3})}^2 .
 \end{equation}
  The result does not hold if $\frac13 \StrichartzConstant^2$ is replaced with a larger constant. 
\end{prop}
\begin{proof}
\begin{figure}[!hbp]
\centering 
\begin{tikzpicture} [
tangent/.style={
        decoration={
            markings,
            mark=
                at position #1
                with
                {
                    \coordinate (tangent point-\pgfkeysvalueof{/pgf/decoration/mark info/sequence number}) at (0pt,0pt);
                    \coordinate (tangent unit vector-\pgfkeysvalueof{/pgf/decoration/mark info/sequence number}) at (1,0pt);
                    \coordinate (tangent orthogonal unit vector-\pgfkeysvalueof{/pgf/decoration/mark info/sequence number}) at (0pt,1);
                }
        },
        postaction=decorate
    },
    use tangent/.style={
        shift=(tangent point-#1),
        x=(tangent unit vector-#1),
        y=(tangent orthogonal unit vector-#1)
    },
    use tangent/.default=1
]
	\draw (0, 0) -- (3, 5)         ;
	\draw (0,0) -- (-3, 5)         ;
	\draw (0, 5) circle (3 and 1/3) ;
	\draw (0, 5) node {\Mrom}; 
	\fill ({atan(5/3)}: {2/3*sqrt(34)}) node (fboldstar) {} node [xshift=-8, yshift=7]  {$c\fboldstar$};
	\draw [help lines, tangent=1] ( {atan(5/3)}: {2/3*sqrt(34)}) arc (0:-135:2/3*3 and 2/3*1/3) node (Gamma) {} node[xshift=12, yshift=10, circle]  {\color{black} $c\Gamma_{\!\alphabold} \fboldstar$} ; 
	\draw [use tangent, rotate=-45, -{>[scale=1.5]}] (0,0) --node[sloped, rotate=-45, above] {\color{black} $\Gamma_{\!\alphabold} \fboldbot$}  (0,2) node (f) {}  node [anchor=north east] {$\fbold$};
	\draw[use tangent, rotate=-45, help lines] (0,2) -- + (1,0) (0, 0) -- ++ (1,0) ++(-0.1, 0) --node[sloped, rotate=-45, above] {\color{black} $d(\fbold, \Mrom)$} +(0, 2); 
	  	
	\draw (0,0) node[anchor=west] {$\obold$};
	\draw [help lines, tangent=0] ( {atan(5/3)}: {2/3*sqrt(34)}) arc (0:-135:2/3*3 and 2/3*1/3) ; 
	\draw [use tangent,rotate=-30, -{>[scale=1.5]}] (0,0) --node[sloped, rotate=-35, above] {\color{black} $\fboldbot$} (0,2) node (GammaMinus) {}  node [anchor=north west] {$\Gamma_{-\alphabold}\fbold$};
	\fill (Gamma) circle (0.05) (f) circle (0.05) (fboldstar) circle (0.05) (GammaMinus) circle (0.05);
	\draw[use tangent, rotate=-30, help lines] 
		(0,2) -- + (1,0) 
		 (1,0) 
		++(-0.1, 0) --node[sloped, rotate=-30, below] {\color{black} $d(\fbold, \Mrom)$} +(0, 2); 
\end{tikzpicture}
 \caption{Illustration of Step 1.}
\label{fig:Distance_Projection}
\end{figure}
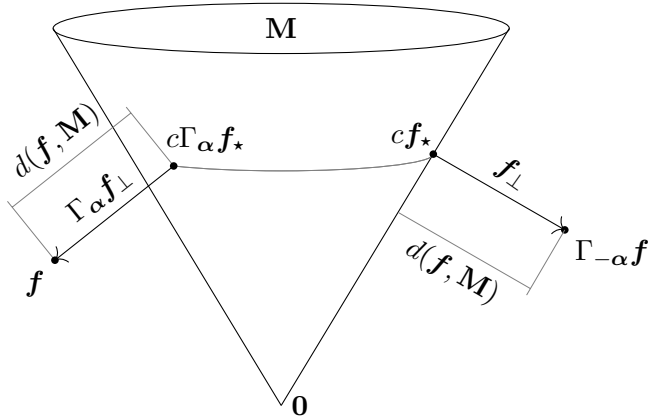
\emph{Step 1}: First, we will show that there exists $c\Gamma_{\!\alphabold}\fboldstar\in \Mrom$, with $c\ne 0$, such that, setting
\begin{equation}\label{eq:GammaFbold}
	\Gamma_{\!\alphabold} \fboldbot := \fbold - c\Gamma_{\!\alphabold} \fboldstar,
\end{equation}
it holds that 
\begin{equation}\label{eq:min_probl_result}
	\begin{array}{ccc}
	 \norm{\fboldbot}_{\Hcaldot^\frac12}=\dist(\fbold, \Mrom) &\text{and}&\fboldbot \bot T_{\fboldstar} \Mrom.
	\end{array}
\end{equation}
To see this, we note that, by the definition-characterization \eqref{eq:Mrom_Manifold} of $\Mrom$,  
\begin{equation}\label{eq:min_probl}
	\dist(\fbold, \Mrom)^2 =\inf\Set{ \norm{\fbold}_{\Hcaldot^\frac12}^2 + c^2\norm{\fboldstar}^2_{\Hcaldot^\frac12} -2c\Braket{ \fbold | \Gamma_{\!\alphabold} \fboldstar}_{\Hcaldot^\frac12} : c, \alphabold		 },
\end{equation}
where $c\in \R, \alphabold\in \SSS^1\times \R^8\times \SO(3)$. We claim that this infimum is attained with $c\ne 0$. Indeed, if $(c_n, \alphabold_n)$ is a minimizing sequence, then $c_n$ is bounded, otherwise $\dist(\fbold, \Mrom)$ would be infinite. Then there is $c\in \R$ such that, up to a subsequence, $c_n\to c$, and by the assumption \eqref{eq:distance_condition}, $c\ne 0$. We now assume by contradiction that $\abs{\alphabold_n}\to \infty$, up to a subsequence. Since $\mathbb S^1$ and $\SO(3)$ are compact, this necessarily implies that one or more of the parameters associated to Lorentz transformations, spacetime translations or dilations must blowup. In the language of~\cite[Lemmas 3.2 and 4.1]{Ramos12}, this means that the sequence $\Gamma_{\!\alphabold_n}$ is \emph{orthogonal} to the identity, and so $\Braket{\fbold | \Gamma_{\!\alphabold_n}\fboldstar}_{\Hcaldot^\frac12} \to 0$. The minimality of $c$ would then imply $c=0$, a contradiction. We conclude that $\alphabold_n$ is bounded, thus it has a convergent subsequence.

Now we define $\fboldbot$ by \eqref{eq:GammaFbold} where $(c,\alphabold)$ attains the minimum in \eqref{eq:min_probl}. As $\Gamma_{\!\alphabold}$ is a unitary operator, the first property in \eqref{eq:min_probl_result} is satisfied. Since $c\ne0$, the tangent space $T_{c\Gamma_{\!\alphabold} \fboldstar}\Mrom$ is well defined, and since $\norm{\fbold- c\Gamma_{\!\alphabold}\fboldstar}_{\Hcaldot^\frac12}^2$ is minimizing, differentiating it we see that 
$$ \Gamma_{\!\alphabold} \fboldbot\,\bot\, T_{c\Gamma_{\!\alphabold} \fboldstar}\Mrom.$$ 
Now, $T_{c\Gamma_{\!\alphabold} \fboldstar}\Mrom = \Gamma_{\!\alphabold} (T_\fboldstar \Mrom)$ by Proposition~\ref{prop:TangentSpaces}, so we can conclude that $\Braket{ \Gamma_{\!\alphabold} \fboldbot | \Gamma_{\!\alphabold} \gbold}_{\Hcaldot^\frac12}=0$ for all $\gbold\in T_\fboldstar \Mrom$. Since $\Gamma_{\!\alphabold}$ is a unitary operator, we infer that $\fboldbot\, \bot\, T_\fboldstar\Mrom$, as claimed.

\emph{Step 2}: Consider the 2-homogeneous deficit functional defined by
\begin{equation}\label{eq:psi_notation}
	\phi(\fbold):=\StrichartzConstant^2\norm{\fbold}_{\Hcaldot^\frac12}^2-\norm{S_t\fbold}_{L^4(\R^{1+3})}^2.
\end{equation}
Like its $4$-homogeneous counterpart $\psi$, the functional $\phi$ is $\Gamma_{\!\alphabold}$-invariant, so that, by Step 1, 
\begin{equation}\label{eq:PreTaylor}
	\phi(\fbold)=\phi(c\Gamma_{\!\alphabold} \fboldstar + \Gamma_{\!\alphabold} \fboldbot) = \phi(c\fboldstar + \fboldbot).
\end{equation}	
Now $\phi(c\fboldstar)=0$, and since $\Braket{\fboldstar|\fboldbot}_\PhaseSpace=0$, we can expand to see that
\begin{equation}\label{eq:phi_first_order}
	\left.\frac{d}{d\eps}\phi(c\fboldstar +\eps \fboldbot)\right|_{\eps=0} = -\frac{2 c}{\norm{S_t\fboldstar}_{L^4}^{2}}\iint_{\R^{1+3}}(S_t\fboldstar)^3S_t\fboldbot\, dtdx.
\end{equation}
Combining Theorem~\ref{thm:put} and Lemma~\ref{lem:Penrose_first_order} from the previous section, we see that the right-hand side is zero. Expanding to second order, using this fact again, we obtain 
\begin{equation}\label{eq:phi_second_order}
	\begin{split}
	\phi(c\fboldstar + \eps \fboldbot)&=\eps^2 \Big[ \StrichartzConstant^2\norm{\fboldbot}_{\PhaseSpace}^2 - \frac3{\norm{S_t\fboldstar}_{L^4}^2}\iint_{\R^{1+3}}(S_t\fboldstar)^2(S_t\fbold)^2\,dtdx \Big]\\ &\ \ \ \ +O(\eps^3\norm{\fboldbot}_\PhaseSpace^3).
	\end{split}
\end{equation}
Evaluating at $\eps=1$, using that $\norm{S_t\fboldstar}_{L^4(\R^{1+3})}=\StrichartzConstant\norm{\fboldstar}_{\PhaseSpace}$, and comparing with the expression of $Q$ given in \eqref{eq:Q_functional}, we obtain
\begin{equation}\label{eq:psi_taylor}
	\phi(c\fboldstar+ \fboldbot) = \frac{Q(\fboldbot) }{2\StrichartzConstant^2\norm{\fboldstar}_{\Hcaldot^\frac12}^2}+ O(\norm{\fboldbot}_{\Hcaldot^\frac12}^3),
\end{equation} 
The proposition then follows from Lemma~\ref{lem:main}, using that $\StrichartzConstant^2=(3/16\pi)^{1/2}$ and that $\norm{\fboldstar}_{\PhaseSpace(\R^3)}^2=\abs{\SSS^3}=2\pi^2$.
\end{proof}

The proof of Theorem \ref{thm:main} will be obtained by the combination of Proposition~\ref{prop:main} with the following property of optimizing sequences of the Strichartz inequality. We remark that, unlike the previous proposition, in the proof of the following lemma we use the result of Foschi that $\StrichartzConstant$ is the sharp constant in the Strichartz inequality.
\begin{lem}\label{lem:structure_lemma}
	Let $\fbold_n\in\Hcaldot^\frac12\setminus\Set{\obold}$ be a sequence such that 
	\begin{equation}\label{eq:maximizing_prop}
		\lim_{n\to \infty} \frac{\norm{S_t\fbold_n}_{L^4(\R^{1+3})}}{\norm{\fbold_n}_{\Hcaldot^\frac12}}=\StrichartzConstant.
	\end{equation}
	Then, up to passing to a subsequence,
	\begin{equation}\label{eq:distance_to_zero}
		\lim_{n\to\infty} \frac{ \dist(\fbold_n, \Mrom)}{\norm{\fbold_n}_{\Hcaldot^\frac12}}=0.
	\end{equation}
\end{lem}
\begin{proof}
By homogeneity we may assume that $\norm{\fbold_n}_{\Hcaldot^\frac12}=1$. We apply the profile decomposition, Theorem \ref{thm:profile_decomposition}. This produces a sequence $\Set{\fbold^j : j\in\N}$ in $\Hcaldot^{1/2}$. We claim that $\fbold^j=0$ for all but one $j\in\N$. To prove this, we begin by showing that there is at least one $j\in\N$ such that $\fbold^j\ne 0$. Indeed, if that was not the case then from property \eqref{eq:Pythagorean_strichartz} one would infer the contradiction $\StrichartzConstant=0$. Thus we can assume that $\fbold^1\ne 0$. 

The Pythagorean expansion \eqref{eq:Pythagorean_energy} with $J=1$ reads
\begin{equation}\label{eq:oneprof_Pyt_energy}
	1=\norm{\fbold^1}_{\Hcaldot^\frac12}^2 +\lim_{n\to \infty} \norm{\rbold^1_n}_{\Hcaldot^\frac12}^2.
\end{equation}
On the other hand, applying the sharp Strichartz inequality to the $L^4(\R^{1+3})$ Pythagorean expansion \eqref{eq:Pythagorean_strichartz} we obtain
\begin{equation}\label{eq:strichartz_to_pythagoras}
	\begin{split}
		\StrichartzConstant^4=\lim_{n\to\infty}\norm{S_t\fbold_n}_{L^4(\R^{1+3})}^4&=\norm{S_t\fbold^1}_{L^4(\R^{1+3})}^4 +\lim_{n\to\infty} \norm{ S_t\rbold^1_n}_{L^4(\R^{1+3})}^4 \\
		&\le\StrichartzConstant^4\Tonde{\norm{\fbold^1}_{\Hcaldot^\frac12}^4 + \lim_{n\to\infty} \norm{\rbold^1_n}_{\Hcaldot^\frac12}^4}.
	\end{split}
\end{equation} 
Now if  $a, b\in \R$ are such that $a^2+b^2=1$ and $a^4+b^4\ge1$, then necessarily one of them must vanish. Since $\fbold^1\ne \obold$, then it must be that $\norm{\rbold^1_n}_{\Hcaldot^\frac12}\to0$. We have thus shown that
\begin{equation}\label{eq:fn_tends_fone}
	\begin{array}{cc}
		\fbold_n=\Gamma_{\!\alpha^1_n}\fbold^1 + \rbold^1_n, & \norm{\rbold^1_n}_{\Hcaldot^\frac12}\to 0.
	\end{array}
\end{equation}
This yields, using \eqref{eq:maximizing_prop}, that $\fbold^1\in\Mrom$. Therefore 
\begin{equation}\label{eq:dist_fbold_Mrom}
	\dist(\fbold_n, \Mrom)\le \norm{\rbold^1_n}_{\Hcaldot^\frac12}\to 0,
\end{equation}
and the proof is complete.
\end{proof}

Combining Proposition~\ref{prop:main} and Lemma~\ref{lem:structure_lemma} we prove the lower bound in Theorem \ref{thm:main}.
\begin{proof}[Proof of Theorem \ref{thm:main}]
Since $\obold\in\Mrom$, we have that
\begin{equation}\label{eq:dist_norm}
	\begin{array}{cc}
		\dist(\fbold, \Mrom)\le\norm{\fbold}_{\Hcaldot^\frac12}, & \forall \fbold\in\Hcaldot^\frac12.
	\end{array}
\end{equation}
Assume for a contradiction that the lower bound of Theorem \ref{thm:main} fails. This would mean that there exists a sequence $\fbold_n\in\Hcaldot^\frac12\setminus\Mrom$ such that 
\begin{equation}\label{eq:seq_contradict}
	\lim_{n\to\infty} \frac{\StrichartzConstant^2\norm{\fbold_n}_{\Hcaldot^\frac12}^2 - \norm{S_t\fbold_n}_{L^4(\R^{1+3})}^2 }{\dist(\fbold_n, \Mrom)^2} = 0.
\end{equation}
By homogeneity we can assume that $\norm{\fbold_n}_{\Hcaldot^\frac12}=1$, and so $\dist(\fbold_n, \Mrom)\le 1$. Then~\eqref{eq:seq_contradict} implies that $\StrichartzConstant^2\norm{\fbold_n}_{\Hcaldot^\frac12}^2- \norm{S_t\fbold_n}_{L^4(\R^{1+3})}^2\to 0$. By Lemma \ref{lem:structure_lemma} we obtain that $\dist(\fbold_n, \Mrom)\to 0$, and so that~\eqref{eq:seq_contradict} would contradict our local bound, Proposition~\ref{prop:main}. 
\end{proof}

\begin{rem}\label{rem:best_constant_lost}
	The multiplicative constant $\frac13\StrichartzConstant^2$ in Proposition \ref{prop:main} is the optimal one for the local bound. However, the argument by contradiction just presented does not give the optimal constant for the global bound. Estimating such optimal constants is in general a hard problem; see for example~\cite[§1.1.1,~§2.5]{Dol21} for a survey of recent progress in the context of Sobolev inequalities. 
	
	\end{rem}

\section*{Acknowledgements} I would like to express my utmost gratitude to my PhD directors, Thomas Duyckaerts and Keith Rogers. They brought these problems and methods to my attention and provided careful guidance and support, without which the writing of this paper would not have been possible. I am also grateful to the editorial board and to the anonymous referees, for the painstaking checking and the numerous suggestions which improved the article.

\printbibliography

\end{document}